%% file: MFG1orNeumann2.tex
\documentclass[10pt,reqno]{amsart}

\usepackage[a4paper,left=35mm,right=35mm,top=30mm,bottom=30mm,marginpar=25mm]{geometry}
\usepackage[backend=bibtex,style=numeric]{biblatex}
\usepackage{amsmath}
\usepackage{amsfonts}
\usepackage{amssymb}
\usepackage{amsthm}
\usepackage{esint}
\usepackage{eurosym}
\usepackage[dvips]{graphics}
\usepackage{graphicx}
\usepackage{epsfig}
\usepackage{enumerate}
\usepackage{verbatim}
\usepackage{hyperref}
\usepackage{dsfont}
\usepackage{bbm}
\usepackage{dsfont}
\usepackage{bm}
\usepackage{mathrsfs}
\usepackage{relsize}
\usepackage{textcomp,enumitem}
\usepackage[leqno]{amsmath}
\usepackage{color}
\input mymacro

\usepackage[displaymath,mathlines]{lineno}
\allowdisplaybreaks
\usepackage{ifthen}
\makeatletter
\makeindex
\makeatother

\numberwithin{equation}{section}

\newtheoremstyle{thmlemcorr}{10pt}{10pt}{\itshape}{}{\bfseries}{.}{10pt}{{\thmname{#1}\thmnumber{
#2}\thmnote{ (#3)}}}
\newtheoremstyle{teolemcorr*}{10pt}{10pt}{\itshape}{}{\bfseries}{.}\newline{{\thmname{#1}\thmnumber{
#2}\thmnote{ (#3)}}}
\newtheoremstyle{defi}{10pt}{10pt}{\itshape}{}{\bfseries}{.}{10pt}{{\thmname{#1}\thmnumber{
#2}\thmnote{ (#3)}}}
\newtheoremstyle{remexample}{10pt}{10pt}{}{}{\bfseries}{.}{10pt}{{\thmname{#1}\thmnumber{
#2}\thmnote{ (#3)}}}
\newtheoremstyle{ass}{10pt}{10pt}{}{}{\bfseries}{.}{10pt}{{\thmname{#1}\thmnumber{
A#2}\thmnote{ (#3)}}}

\theoremstyle{thmlemcorr}
\newtheorem{theorem}{Theorem}
\numberwithin{theorem}{section}
\newtheorem{lemma}[theorem]{Lemma}

\newtheorem{teo}[theorem]{Theorem}
\newtheorem{lem}[theorem]{Lemma}
\newtheorem{pro}[theorem]{Proposition}
\newtheorem{cor}[theorem]{Corollary}

\theoremstyle{thmlemcorr*}
\newtheorem{theorem*}{Theorem}
\newtheorem{lemma*}[theorem]{Lemma}
\newtheorem{corollary*}[theorem]{Corollary}
\newtheorem{proposition*}[theorem]{Proposition}
\newtheorem{problem*}[theorem]{Problem}
\newtheorem{conjecture*}[theorem]{Conjecture}

\theoremstyle{defi}
\newtheorem{definition}[theorem]{Definition}
\newtheorem{hyp}{Assumption}

\newtheorem{problem}{Problem}

\theoremstyle{remexample}
\newtheorem{remark}[theorem]{Remark}

\theoremstyle{ass}

\begin{document}

\title[Time dependent first-order MFG with Neumann Conditions]{Time dependent first-order Mean Field Games with Neumann boundary conditions}

\author{Diogo A. Gomes}
\address[D. A. Gomes]{
King Abdullah University of Science and Technology (KAUST), CEMSE Division, Thuwal 23955-6900. Saudi Arabia.}
\email{diogo.gomes@kaust.edu.sa}
\author{Michele Ricciardi}
\address[M. Ricciardi]{
Università LUISS - Guido Carli. Viale Romania 32, 00197 Roma, Italy.}
\email{ricciardim@luiss.it}

\keywords{Mean Field Games; Neumann conditions; }
\subjclass[2010]{
35J47, 
35A01} 

\thanks{The authors were supported by King Abdullah University of Science and Technology (KAUST) baseline funds and KAUST OSR-CRG2021-4674.
}
\date{\today}

\begin{abstract}
	The primary objective of this paper is to understand first-order, time-depen\-dent mean-field games with Neumann boundary conditions, a question that remains under-explored in the literature. This matter is particularly relevant given the importance of boundary conditions in crowd models. 
	 In our model, the Neumann conditions result from 
	players entering the domain $\Omega$ according to a prescribed current $j$, for instance, in a crowd entry scenario into an open-air concert or stadium. 
	We formulate the model as a standard mean-field game coupling a Hamilton-Jacobi equation
	with a Fokker-Planck equation. Then, we introduce a relaxed variational problem and use Fenchel-Rockafellar duality to study the relation between these problems. Finally, 
	we prove the existence and uniqueness of solutions for the system using variational methods.
\end{abstract}

\maketitle

\section{Introduction}

\input{Introduction.tex}



\section{Preliminary material and Main Assumptions}\label{sec2}
\input{Assumptions.tex}

\section{The variational formulation}\label{sec3}
\input{Variational.tex}

\section{Duality}\label{sec4}

\input{Duality.tex}

\section{A relaxation of the problem}\label{sec5}

\input{Relaxation.tex}

\section{Existence of a minimizer for the relaxed Problem}\label{sec6}
\input{Existence.tex}

\section{Existence of solutions}\label{sec7}

\input{MFG.tex}

\section{Uniqueness and Main Theorem}\label{sec8}
\input{uniqueness.tex}

\subsection*{Authors' contributions :} Both authors contributed equaly for this publication. 

\subsection*{Data availability}
The relevant materials can be accessed through the citations provided in this paper. Our research is primarily theoretical and mathematical in nature, and as such, we did not generate or use any code or datasets.

\subsection*{Ethical Statement} This work focuses on mathematical models and does not involve any human or animal subjects. Therefore, no ethical approval was required.

\subsection*{Competing interests:} The authors declare no competing interests.


\printbibliography
\end{document}

%% file: mymacro.tex
\newcommand{\tr}{\operatorname{tr}}

\newcommand{\supp}{\operatorname{supp}}

\newcommand{\Rr}{{\mathbb{R}}}

\newcommand{\bx}{{\bf x}}

\newcommand{\bdx}{\dot{\bf x}}

\def\leq{\leqslant}
\def\geq{\geqslant}

\def \N{\mathbb{N}}
\def \R{\mathbb{R}}

\def\vfi{\varphi}

\newcommand{\leqnomode}{\tagsleft@true}
\newcommand{\reqnomode}{\tagsleft@false}

\newcommand{\miezz}{\frac{1}{2}}

\newcommand{\eps}{\varepsilon}

\newcommand{\into}{\ensuremath{\int_{\Omega}}}
\newcommand{\intf}{\ensuremath{\int_{t}^{T}\int_{\Omega}}}
\newcommand{\inti}{\ensuremath{\int_{0}^{t}\int_{\Omega}}}
\newcommand{\intif}{\ensuremath{\int_{0}^{T}\int_{\Omega}}}

\newcommand{\norm}[1]{\ensuremath{\left\Arrowvert #1 \right\Arrowvert}}
\newcommand{\norminf}[1]{\ensuremath{\left\Arrowvert #1 \right\Arrowvert_\infty}}

\newcommand{\weak}{\ensuremath{\rightharpoonup}}

\newcommand{\intb}{\int_{0}^{T}\int_{\partial\Omega}}

\newcommand{\bo}[1]{\boldsymbol{#1}}

%% file: Introduction.tex
Mean Field Game (MFG) theory, introduced by Lasry and Lions \cite{ll1,ll2,ll3} and independently by Caines, Huang, and Malham\'e \cite{Caines1}, 
describes non-cooperative differential games with numerous agents,  finding applications in diverse fields such as economics, finance, crowd dynamics, and social sciences. 
When the number of agents is large, modeling individual strategies becomes impractical. To address this, we use a macroscopic approximation of the game, leading to a system of partial differential equations (PDEs).
A key feature of this framework is that it ensures each agent acts rationally by accounting for their personal goals and the expected behavior of the entire population. 
Agents optimally choose their actions by anticipating future costs and considering how the actions of others will impact the system.
For instance, agents may avoid congested areas to minimize travel costs or choose quicker exits when fewer agents are nearby.
This system comprises a Hamilton--Jacobi equation for the value function $u$ associated with the cost functional of the generic agent and a transport or Fokker--Planck equation for the law of the population, $m$.

This paper provides the first rigorous examination of first-order MFGs with Neumann boundary conditions in a time-dependent setting, addressing a significant gap in the literature. 
Boundary conditions play a crucial role in crowd modeling, particularly in determining how individual agents interact with the environment at domain edges. In real-world scenarios, such as the flow of people into or out of a venue, these conditions are essential for accurately representing crowd dynamics. They determine factors like entry and exit rates, restrictions on movement in certain areas, and population behavior at physical barriers, entrances, or exits.

A substantial portion of the theoretical work in MFG research has been dedicated to topics such as well-posedness, solution regularity, and long-time behavior. While much of the literature focuses on the periodic setting, boundary conditions, including Dirichlet, reflecting, invariance, and state constraint conditions, are vital in applications.

Reflecting conditions and state constraint conditions both deal with situations where agents cannot leave the domain $\Omega$, but they differ in their implementation and implications. Reflecting conditions model scenarios where agents encounter a physical barrier at the boundary, such as a wall, or face economic or financial constraints that cause them to "bounce back" into the domain, see eg \cite{moll}. In contrast, state constraint conditions arise when agents are simply not allowed to leave the domain, without specifying a reflection mechanism. These were studied in the  first-order setting in \cite{CaCaCa2018,CAPUANI2022180,CapMarRic}, and in the ergodic second-order case in \cite{CannMend,PorRic-ergodic}. Another natural possibility is the Dirichlet case, which corresponds to exit costs at the boundary, as explored in \cite{CampiFischer}.

The existence and uniqueness of solutions for second-order cases were examined not only in periodic settings (e.g., $\Omega=\mathbb T^N$ in \cite{GPM2,GPM3} for subquadratic and superquadratic cases), but also under Dirichlet or Neumann boundary conditions. For example, see \cite{porretta2} for an in-depth exploration of these conditions with minimal assumptions on the data, \cite{MR3333058} for Neumann boundary conditions in multi-population MFG, and \cite{FestaRicciardi} for forward-forward MFGs with Dirichlet or Neumann boundary conditions.

The study of the Master Equation and the convergence problem in MFG theory under Dirichlet and Neumann conditions has been explored in \cite{MatteoMichele,RicciardiConvergence,RicciardiME}. Invariance conditions, which prevent agents from leaving the domain independent of their control strategies, are examined in \cite{MR4045803}.

For first-order MFGs, the situation is more complex due to the lack of a regularizing effect from a second-order term. For the periodic first-order case, $\Omega=\mathbb T^N$, see \cite{Card1order}; for less restrictive assumptions, see \cite{CardGrab}. Unlike second-order cases, where boundary conditions for $u$ and $m$ can be specified (Dirichlet or Neumann), first-order MFGs require different boundary conditions because using those from second-order models leads to overdetermined systems. As a result, either $u$ or $m$ must be prescribed at the boundary. This issue is discussed in the stationary first-order MFG studied in \cite{alharbi2023firstorder}, where the authors introduce more general mixed boundary conditions. In that work, the coercivity of the Hamiltonian allows for compactness estimates of the value function.

Building on these foundations, this paper aims to address the lack of results for first-order MFGs with Neumann boundary conditions in a time-dependent setting. 
A fundamental difficulty  is the absence of a compactness result due to the term $-u_t$ that precludes using the methods
developed for the stationary case. 
To our knowledge, this is the first work to examine this problem rigorously.

For instance, consider crowd entry into venues like open-air concerts or stadiums. In these situations, agents enter the domain (the venue) at a specific rate, which can be modeled using Neumann boundary conditions to represent the inflow of population density. By adjusting these boundary conditions, we can model various real-world crowd scenarios, such as managing crowd control or optimizing the flow of people into a stadium. In particular,  
we examine the following problem.

\begin{problem}\label{problem_mfg}
	Let $\Omega\subset\R^N$ be an open bounded domain with a Lipschitz boundary $\partial\Omega$. Let $T>0$ denote the terminal time of the game, and let $Q_T=[0,T]\times\Omega$. Consider the \emph{parabolic boundary} $\partial_p Q_T=(\{0,T\}\times\Omega)\cup([0,T]\times\partial\Omega)$.
	
	Let $f:[0,T]\times\Omega\times\R_0^+\to\R$, {$H:\Omega\times\R^N\to\R$}, $\psi:\Omega\to\R$, $j:[0,T]\times\partial\Omega\to\R$ be smooth functions. Let $\nu$ be the outward normal at the boundary $\partial\Omega$.
	
	Find $u\in L^1([0,T];W^{1,1}(\Omega))$ and $m\in L^1([0,T]\times\Omega)$ such that $(u,m)$ solves 
	\begin{equation}\label{mfg}
		\begin{cases}
			-u_t+ {H(x,\nabla u)}=f(t,x,m), & (t,x)\in(0,T)\times\Omega,\\
			m_t-\mathrm{div}(m {H_p(x,\nabla u)})=0, & (t,x)\in(0,T)\times\Omega,\\
			m(0,x)=m_0(x),\qquad u(T,x)=\psi(x), & x\in\Omega,\\
			m {H_p(x,\nabla u)}\cdot\nu(x)=j(t,x), & (t,x)\in(0,T)\times \partial\Omega,\\
			m(t,x)\ge0, & (t,x)\in(0,T)\times\Omega\text{ a.e.}.
		\end{cases}
	\end{equation}
\end{problem}

The Neumann condition plays a dual role in our model, determining the flow of population density, $m$, at the boundary $\partial \Omega$, while also acting as a state constraint for the value function $u$ in the Hamilton--Jacobi equation. Assumption \ref{hp3} (detailed in Section \ref{sec2}) imposes a strictly positive current, $j$, ensuring a positive inflow rate at every boundary point. This condition manifests as a negative normal component on the agents' velocity:
$$
\bdx=-H_p(\bx,\nabla u(t,\bx)).
$$
By restricting the normal velocity of agents at the boundary, the Neumann condition influences both the distribution of the population and the agents' paths.  In this framework, only entry into the domain $\Omega$ is permitted, with exit being prohibited.  
Hence, prescribing a Dirichlet condition for the value function is unnecessary.

In the preceding problem, we assumed smoothness on all data. It is clear from the proof that much less regularity is needed, but assuming smoothness allows us to focus only on the main technical aspects. In the various assumptions, however, we detail which smoothness properties are critical (c.f. Assumption \ref{hp} in Section \ref{sec2}).

The following theorem asserts the existence and uniqueness of solutions for the Problem \ref{problem_mfg}.
The proof uses a variational formulation that we develop
in this paper, see Problems \ref{var1} and \ref{var2}. 

\begin{teo}\label{thm:main}
	Consider the setting of Problem \ref{problem_mfg}.
	Suppose Assumptions \ref{hp}, \ref{hp2}, \ref{hp3} and \ref{hp4} hold. 
	Then Problem \ref{problem_mfg} has a solution $(u,m)\in L^1([0,T];W^{1,1}(\Omega))\times L^1([0,T]\times\Omega)$ in the sense of Definition \ref{defmfg}. 
	Moreover, if $(u_1,m_1)$ and $(u_2,m_2)$ are two solutions of
	Problem \ref{problem_mfg}, then $m_1=m_2$ a.e., and $u_1=u_2$ a.e.
	in the set $\{m_1>0\}$.
\end{teo}

Section \ref{sec2} outlines the assumptions required for the preceding theorem and the 
precise definition of solution. 
These assumptions include growth and coercivity assumptions (Assumption \ref{hp}), positivity conditions
for $m_0$ and $j$ (Assumption \ref{hp2}), and additional technical conditions on the domain and $H$ (Assumption \ref{hp3}). In particular, Assumption \ref{hp3} requires $\Omega$ to be a rectangular domain. This requirement plays a crucial role in the proof of Lemma \ref{lem:regularization}; 
all other results are valid if $\Omega$ is a bounded open set with a Lipschitz boundary $\partial\Omega$.

The proof is based on a variational formulation developed in Section \ref{sec3}, where we examine corresponding dual problems and explicitly connect them to the MFG system’s Euler-Lagrange optimality conditions.
To formulate these variational problems, we define
\begin{equation}\label{def:F}
	F(t,x,m):=\left\{\begin{array}{lr}
		\displaystyle\int_0^m f(t,x,s)ds & \text{if }m\ge0,\\
		+\infty & \text{if }m<0.
	\end{array}\right.
\end{equation}
The convex conjugate of  $F$ with respect to the last argument,  $F^*:[0,T]\times\Omega\times\R\to\R\cup\{+\infty\}$, is
\begin{equation}
	\label{legtransF}
	F^*(t,x,s)=\sup\limits_{m\geq 0}\{sm-F(t,x,m)\}.
\end{equation}
Similarly, for the Hamiltonian,  {$H:\Omega\times\R^N\to\R$}, we define  {$H^*:\Omega\times\R^N\to\R\cup\{+\infty\}$} as 
\begin{equation}
	\label{legtransH}
	{H^*(x,p^*)=\sup\limits_{p\in\R^N}\{p\cdot p^*-H(x,p)\}};
\end{equation}
the Lagrangian,  {$L:\Omega\times\R^N\to\R\cup\{+\infty\}$, is $L(x,p^*):=H^*(x,-p^*)$}.
The two following variational problems below are examined in 
Section \ref{sec3}. 
\begin{problem}\label{var1}
	Consider the setting of Problem \ref{problem_mfg}. 
	Let 
	\begin{equation}\label{eq:Iset}
		\mathscr I:=\{v\in\mathcal{C}^1(Q_T)| v(T,x)=\psi(x)\}.
	\end{equation}
	Find $v\in\mathscr I$ minimizing the functional
	$$
	\mathcal I(v):=\intif F^*(t,x,-v_t+ {H(x,\nabla v)})dxdt-\into v(0,x)m_0(x) dx-\int_0^T\int_{\partial\Omega}jvdx dt.
	$$
\end{problem}
As shown in Proposition \ref{pro_1promfg}, the optimality conditions of the previous problem are equivalent 
to \eqref{mfg}. The dual problem to Problem \ref{var1} is as follows.
\begin{problem}\label{var2}
	Consider the setting of Problem \ref{problem_mfg}. 
	Let $\mathscr M\subset L^1([0,T]\times\Omega;\R)\times L^1([0,T]\times\Omega;\R^N)$ be the set of all pairs $(m,w)$ such that $m\in L^1([0,T]\times\Omega;\R)$, $w\in L^1([0,T]\times\Omega;\R^N)$, $m\ge 0$, and $(m,w)$ solves, in a distributional sense,
	\begin{equation}\label{eq:fp}
		\begin{cases}
			m_t+\mathrm{div}(w)=0, & (t,x)\in(0,T)\times\Omega,\\
			m(0,x)=m_0(x), & x\in\Omega,\\
			-w(t,x)\cdot\nu(x)=j(t,x), & (t,x)\in(0,T)\times \partial\Omega; \\
		\end{cases}
	\end{equation}
	that is, 
	for all $\xi\in W^{1,\infty}(Q_T)$ with $\xi(T,\cdot)=0$, we have
	\begin{equation}\label{either}
		\intif (m\xi_t+w\cdot\nabla\xi)dxdt+\into \xi(0,x) m_0(x)dx+\int_0^T\int_{\partial\Omega} j\xi dx dt=0.
	\end{equation}
	
	Find $(m,w)\in\mathscr M$ minimizing the functional
	\begin{align}
		\mathcal M(m,w):=&\intif \left[m {L\left(x,\frac wm\right)}+F(t,x,m)+w\cdot\nabla \psi(x)\right]dxdt\notag\\
		+&\into \psi(x)m_0(x)dx+\int_0^T\int_{\partial\Omega}\psi(x)j(t,x)dx dt\label{def:M},
	\end{align}
	where we adopt the following convention: if $m(t,x)=0$ for some $(t,x)\in Q_T$, then
	$$
	m(t,x) {H^*\left(x,\frac {w(t,x)}{m(t,x)}\right)}=\begin{cases}
		+\infty &\hbox{ if $w\neq 0$},\\
		0 &\hbox{ if $w=0$}.
	\end{cases}
	$$
\end{problem}

The  Fenchel-Rockafellar duality theorem 
aids in understanding the duality between Problems \ref{var1} and \ref{var2}.
These duality relations have been explored previously in the context of 
Hamilton--Jacobi equations, see \cite{EGom3}, \cite{MR2883292},\cite{MR2354987},  \cite{gomes2020large}, for example.
As detailed in Section \ref{sec4}, 
we build on the ideas developed in \cite{Card1order} to establish the duality result 
in Theorem \ref{duality}. According to this theorem,  Problem \ref{var2} has a minimum $(m,w)$ with $m\in L^q$ for some $p>1$. This is not necessarily the case for Problem \ref{var1}; due to the non-coercivity of $-u_t+ {H(x,\nabla u)}$, this problem may not admit a minimum in the set $u\in\mathcal C^1([0,T]\times\overline{\Omega})$. 
To address the potential failure of Problem \ref{var1} to have a solution, we introduce a relaxed version in Section \ref{sec5}, where we expand the admissible set to preserve duality with Problem \ref{var2}, allowing us to recover existence results and establish a connection between the original and relaxed problems.
A similar idea was used in \cite{Card1order} and \cite{CardGrab} for coercive Hamiltonian $H$ and cost function $f$, in \cite{Graber} for linearly bounded Hamiltonians, \cite{Tono2019} for the planning problem. 
The problems in all of these references were studied under periodic boundary conditions.
Here,
the Neumann conditions present the main technical difficulty.

By approximating subsolutions of the Hamilton--Jacobi equation with smooth functions, we establish the duality between Problem \ref{var3} and Problem \ref{var2}, see Theorem \ref{thm:duality}.
Next, in Section \ref{sec6},  we show the existence of a minimizer for Problem \ref{var3}.
We do this by considering a minimizing
sequence $v^\eps$ for Problem \ref{var2} and establishing
its convergence within the enlarged space $\mathscr K$ to the minimum. 
For this, we need a careful study of the convergence at the boundary to ensure the convergence of the boundary integral in $\mathcal I$.
Next, in Section \ref{sec7}, 
we show that minima of Problems  \ref{var3} and \ref{var2} provide solutions of Problem \ref{problem_mfg}. Moreover, any solution to Problem \ref{problem_mfg} yields minimizers 
for Problems  \ref{var3} and \ref{var2}. Lastly, Section \ref{sec8} establishes the
uniqueness of solutions to Problem \ref{problem_mfg} and ends the proof of Theorem \ref{thm:main}.	 	 

To summarize, this paper makes several key contributions to MFG theory. We provide the first rigorous examination of first-order MFGs with Neumann boundary conditions in a time-dependent setting, addressing a significant gap in the literature. To overcome the lack of compactness results typically used in stationary cases, we develop a variational formulation and introduce a relaxed version that preserves duality with the dual problem. Then, we establish the existence and uniqueness of the solutions to this MFG system. Our approach extends the application of MFG theory to more realistic scenarios in crowd modeling, where boundary conditions play a crucial role in determining agent behavior at domain edges. These advancements contribute to the theoretical foundations of MFG theory and provide a framework for modeling a more comprehensive range of real-world phenomena, particularly in the context of crowd dynamics and population flow.

%% file: Assumptions.tex
This section recalls essential properties of the convex conjugate, which are fundamental for our subsequent analysis. We refer the reader to \cite{Ekeland} for additional details. These properties are frequently used throughout the paper as we often work with the convex conjugates of the Hamiltonian $H$ and the primitive $F$ of the coupling function $f$.
We also present the main assumptions for our analysis. These include standard growth and coercivity conditions (Assumption \ref{hp}) and conditions on the initial data and agent flow rate (Assumption \ref{hp2}). Next, we present a technical assumption that simplifies the study of Neumann boundary conditions (Assumption \ref{hp3}).  Finally, we provide the definition of a weak solution for Problem \ref{problem_mfg}.

\subsection{Preliminary results about the convex conjugate}
Let $A\subset\R^N$ be a convex set and $\Phi:A\to\R$ be a lower semicontinuous and convex function.
We define the convex conjugate of $\Phi$ as in the Introduction, i.e., $\Phi^*:\R^N\to\R\cup\{+\infty\}$ is given by
$$
\Phi^*(p):=\sup\limits_{x\in A}\{x\cdot p-\Phi(x)\} .
$$
We say that $p$ and $x^*$ are conjugated variables if 
$\Phi^*(p):=x^*\cdot p-\Phi(x^*)$ .
The convexity and lower semicontinuity of $\Phi$ imply that $\Phi^{**}=\Phi$. Moreover, assume that $\Phi$ is differentiable and strictly convex, i.e.
$$
\Phi(\lambda x+(1-\lambda y))< \lambda\Phi(x)+(1-\lambda)\Phi(y) ,\qquad\forall\lambda\in(0,1) ,\quad\forall x,y\in A\mbox{ s.t. }x\neq y .
$$
Then, we have that $\Phi^*$ is differentiable and
	$$
	\Phi^*_p(p)=\arg\max\limits_{x\in A}\{\langle x,p\rangle-\Phi(x)\}. 
$$
Moreover, $\Phi^*_p(p)$ and $p$ are conjugated variables and
	\begin{equation}\label{hp:CC}
		\Phi(\Phi^*_p(p))+\Phi^*(p)=\Phi^*_p(p)\cdot p , \qquad \Phi^*_p(\cdot)=\Phi_x^{-1}(\cdot) .
	\end{equation}

Furthermore, we note that strict convexity of $\Phi$ does not guarantee the same for the conjugate $\Phi^*$.
Consider, for instance, the function
$$
\Phi(x)=\begin{cases}
	x^2 & x<1 ,\\
	x^3 & x\ge 1 ,
\end{cases}
$$
and its conjugate
$$
\Phi^*(p)=\begin{cases}
	\frac{p^2}4 & p<2 ,\\
	p-1 & 2\le p<3 ,\\
	\frac{2\sqrt 3}9 p\sqrt p & p\ge 3 .
\end{cases}
$$
Although $\Phi $ is strictly convex (but not $\mathcal C^1$), its conjugate,  $\Phi^*$, is not strictly convex in the region  $2\le p<3$. 
The following lemma shows that imposing additional $C^1$ regularity for $\Phi$ gives the strict convexity of its conjugate. 

\begin{lemma}\label{lem:strictHstar}
	Let $N\ge1$ and assume that $\Phi:\R^N\to\R\cup\{+\infty\}$ is a strictly convex and differentiable function. Then $\Phi^*:\R^N\to\R\cup\{+\infty\}$ is a strictly convex function.
\end{lemma}
The lemma can be generalized to functions $\Phi:A\to\R^N$ on a convex domain $A\subset\R^N$ by assigning $\Phi(x)=+\infty$ for $x$ outside $A$. 

\begin{proof}
	The proof is straightforward using the preceding results about the convex conjugate. Since $\Phi$ is strictly convex, we have $\nabla \Phi$ strictly monotone, in the sense that
	$$
	\langle\nabla \Phi(x)-\nabla \Phi(y),x-y\rangle>0\qquad\text{for }x, y\in\R^N ,\quad x\neq y .
	$$
	Hence, $\nabla \Phi^{-1}$ is strictly monotone too. Since $\Phi$ is differentiable, we have $\nabla \Phi^*=\nabla \Phi^{-1}$, which ensures the strict convexity of $\Phi^*$.
\end{proof}

\subsection{Notation and Assumptions}
Here, we work in the setting of Problem \ref{problem_mfg}. In particular,  
$\Omega\subset\R^N$ is a bounded Lipschitz domain, $Q_T=[0,T]\times\Omega$, 
and for $t\in(0,T)$, we define $Q_t:=[0,t]\times\Omega$. Furthermore, 
throughout the paper, for $n>1$, $n'$ denotes the conjugate of $n$, given by $n':=\frac n{n-1}$. 

Regarding the data in Problem \ref{problem_mfg}, namely $m_0$ and $j$, as well as the functions 
that determine the MFG, $H$ and $f$, we work under the following assumption. 
\begin{hyp}\label{hp}
	There exist $q>1$, $r>1$, and $C>0$ such that 
	\begin{enumerate}
		\item\label{item1} $m_0\in L^1(\Omega)$, with $m_0\ge0$ almost everywhere, whereas $\psi\in\mathcal C^1(\overline\Omega)$, $j\in\mathcal C([0,T]\times\partial\Omega)$, with $j\ge0$;
		\item\label{item2} {$H:\Omega\times\R^N\to\R$ is continuous in both variables. Moreover, it is strictly convex and differentiable in the last variable, and for all $(x,p)\in\Omega\times\R^N$},
		$$
		C^{-1}|p|^r-C\le {H(x,p)}\le C(|p|^r+1);
		$$

		\item\label{item3} $f:[0,T]\times\Omega\times(0,+\infty)\to\R$ is continuous, strictly increasing in the last argument, and, for all $(t,x,m)\in[0,T]\times\Omega\times[0,+\infty)$,
		$$
		C^{-1}m^{q-1}-C\le f(t,x,m)\le C(m^{q-1}+1);
		$$
	\end{enumerate}
\end{hyp}

The conditions $m_0\geq 0$ and $j\geq 0$ in Item~\ref{item1} are natural given their interpretation: $m_0$ is the initial distribution of agents and $j$ corresponds to the inflow rate into the domain $\Omega$. 
The growth requirements on $f$ and $H$ in Items~\ref{item2} and \ref{item3} are natural in our variational setting. They are essential to obtain the energy estimates necessary for the convergence of the minimizing sequences for the variational problems mentioned in the introduction; see Proposition \ref{pro:ex_relax}. Similar growth conditions are often used in MFGs; see, for example, the conditions for the parabolic case in
\cite{GPM2} or \cite{GPM3}. For weak solutions, for parabolic MFGs, the condition on $f$ can be somewhat relaxed; see \cite{MR3691806}, for example. 
Note that the quadratic Hamiltonian, ${H(x,p)}=\miezz|p|^2$, 
and the coupling $f(m)=m^\alpha$, 
satisfy the preceding assumption.
In the MFGs, the monotonicity of $f$ suggests that players avoid areas of high density. This hypothesis is often used in the literature to obtain the uniqueness of solutions for Problem \ref{problem_mfg}, as is the case here to prove the uniqueness of solutions for Problem \ref{var2}.

Assumption \ref{hp} yields the following estimates on $F$, $F^*$, and $H^*$.

\begin{remark}\label{rem21}
	If Assumption \ref{hp} holds, the function $H^*$ inherits the same regularity of $H$; moreover, for a possibly different constant $C$,
	$$
	C^{-1}|p^*|^{r'}-C\le {H^*(x,p^*)}\le C\left(|p^*|^{r'}+1\right).
	$$
	Similarly, the function $F$ defined previously is continuous, differentiable, and convex in the last argument, and the same applies for $F^*$. Moreover, for all $m\ge 0$ and $s\ge 0$,
	\begin{equation}\label{hp:FFstar}
		C^{-1}m^q-C\le F(t,x,m)\le C(m^q+1) ;\qquad C^{-1}s^{q'}-C\le F^*(t,x,s)\le Cs^{q'}+C.
	\end{equation}
	Due to Lemma \ref{lem:strictHstar},  $H^*$ and $F^*$ are strictly convex because $H$ and $F$ are differentiable.
\end{remark}

Because $H$ and $f$ are bounded from below, we can add a constant to both sides of the Hamilton-Jacobi equation in \eqref{mfg} (e.g., $c_1:=|\inf f|+|\inf H|$) without altering the solution set of the system. 
Hence, by defining $\tilde f(t,x,m)=f(t,x,m)+c_1$, {$\tilde H(x,p)=H(x,p)+c_1$}, we can assume, without loss of generality, that $f$ and $H$ are non-negative. Thus, the primitive function $F$, defined in \eqref{def:F}, is non-decreasing in the variable $m$, for $m>0$.

\begin{remark}\label{rem23}
	Since $F=+\infty$ for $m<0$, the supremum in \eqref{legtransF} must be attained for $m\ge0$. This implies that $F^*$ is non-decreasing in the last variable. 
	Moreover, because $f$ is non-decreasing in the last variable, we have $F(t,x,m)\ge f(t,x,0)m$ for $m\ge0$. In particular, 
	we have $F\ge0$ for $m\geq 0$. Finally,  $F^*(t,x,s)\le\sup\limits_{m\ge0}(s-f(t,x,0))m$.
	Accordingly
	\begin{equation}\label{vanish_Fstar}
		F^*(t,x,s)=0\qquad\text{for }s\le f(t,x,0) 
	\end{equation}
	and, using the fact that $F^*$ is non-decreasing in the last variable, 
	$F^*\ge0$ for $s> f(t,x,0) $
\end{remark}
The preceding property is used in Section \ref{sec7}, Theorem \ref{teo:ex_mfg}  to prove the existence of
solutions for Problem \ref{problem_mfg}.

To prove some of the results in this paper, we require additional assumptions on the coefficients, the boundary data, and the domain structure. 

The following Assumption is of a technical nature as it prescribes
that $\Omega$ must be a rectangular domain and
requires an additional evenness hypothesis on the Hamiltonian $H$.

\begin{hyp}\label{hp2}
	We have
	\begin{enumerate}[start=4]
		\item\label{item5} $\Omega$ is a rectangular domain, i.e., for $1\le i\le N$ there exist $a_i,b_i$ with $a_i<b_i$ and $\Omega=\prod\limits_{i=1}^N(a_i,b_i) $;
		\item\label{item6} {$H:\Omega\times\R^N\to\R$ is even in all variables $p_i$,} i.e., for all $p=(p_j)_j\in\R^N$ and for all $i\in\N$ such that $1\le i\le N$, we have
		$$
		{H(x,p_1,\dots,p_{i-1},p_i,p_{i+1},\dots,p_N)=H(x,p_1,\dots,p_{i-1},-p_i,p_{i+1},\dots,p_N)\,,\qquad\forall\,x\in\Omega.}
		$$
		{\item\label{item6.5} We assume there exists $\theta>0$, $0<\varsigma\le 1$, with $\theta(N+1)\le r\varsigma$ and such that, for a certain constant $C>0$, we have
			$$
			|H(x,\xi)-H(y,\xi)|\le C(1+|\xi|)^\theta|x-y|^\varsigma\,,\qquad\forall\,x,\,y\in\Omega,\,\,\xi\in\R^d\,;
			$$}
		\item\label{item7} the following relation between the growth exponents: $r>N(p-1)$.
	\end{enumerate}
\end{hyp}

The evenness of $H$ in Item~\ref{item6} can be relaxed. For example, our results extend to a Hamiltonian of the form ${H(x,p)=H_1(x,p)+H_0(x,p)}$, 
where $H_1$ is even and satisfies Assumption \ref{hp}, and $H_0\in L^\infty$. 
We do not include this extension here for simplicity and to avoid too complex notation.
We use Items \ref{item5}, \ref{item6}, and \ref{item6.5} in Lemma \ref{lem:regularization} to extend the subsolutions of the Hamilton-Jacobi equation beyond $\Omega$ through reflection. The evenness of $H$ and the rectangular domain ensure the validity of these reflected subsolutions. Item \ref{item6.5} is crucial for demonstrating that we retain control over the approximation error when we approximate subsolutions to the Hamilton-Jacobi-Bellman equation (HJB) by convolution with a smooth kernel.
This hypothesis slightly generalizes the one in  \cite{Card1order}. That hypothesis was no longer required in \cite{CardGrab}, since the approximation by convolution was done in the Fokker-Planck equation.  However, the Neumann boundary condition in our case complicates such an approximation, and our method requires using
Item \ref{item6.5}. The condition on $r$ in Item~\ref{item7} is used to demonstrate a partial H\"older estimate for Problem \ref{var3}, see Lemma \ref{lem:holder} using techniques similar to the ones in \cite{Card1order,CardGrab}.

The following Assumption is on the boundary data $m_0$ and $j$.

\begin{hyp}\label{hp3}
	We have
	\begin{enumerate}[start=7]
		\item\label{item8} The data $m_0$ and $j$ are strictly positive, i.e. $m_0>0$ and $j>0$.
	\end{enumerate}
\end{hyp}
An initial distribution $m_0$ strictly positive was already required in \cite{CardGrab}, allowing us to prove the existence of solutions for the relaxed Problem \ref{var3}. We require the same strict positivity hypothesis on $j$,  which plays the same role in Proposition \ref{pro:ex_relax}. 
A nonvanishing current at the boundary ensures the state constraint condition for the value function $u$. 
When $j>0$, we have ${H_p(x,\nabla u)}\cdot\nu>0$ at $\partial\Omega$. 
Hence, the agent’s velocity points inside the domain when the players are near the boundary.
This is precisely the state constraint that ensures the confinement of the process into $\Omega$. 
See also the second-order case's invariance condition stated in  \cite{BCR,CDF,CCR,MR4045803}.

The last Assumption is a strengthening hypothesis on the growth of the coefficients' exponents.
\begin{hyp}\label{hp4}
	We have
	\begin{enumerate}[start=8]
		\item\label{item9} The following relation between the growth exponents is satisfied: $r>\frac{Nq}{Nq+q-2}$.
	\end{enumerate}
\end{hyp}

Item~\ref{item9} is a technical hypothesis used here  to prove the uniqueness of solutions for the relaxed problem and, consequently, for the Mean Field Game system.
This hypothesis is implied in the bound for $r$ in Item~\ref{item7} of Assumption ~\ref{hp2} if $q>1+\frac{\sqrt{N+1}}{N+1}$.

An example of a MFG satisfying all these Assumptions is determined by 
$$
H(x,p)=c\left(1+|p|^{\frac r2}\right)^2{+\,\varphi(x)} ,\qquad f(t,x,m)=\tilde f(t,x)(1+m^q) ,
$$
where $c>0$, $\tilde f$ is continuous with $\tilde f(t,x)>\eps>0$ for all $(t,x)\in Q_T${, and $\varphi$ is a Lipschitz function.}

\subsection{Definition of solution}

Here, we define weak solutions for Problem \ref{problem_mfg}.
\begin{definition}\label{defmfg}
	A pair $(u,m)\in BV(Q_T)\times L^q(Q_T)$ is a weak solution of Problem \ref{problem_mfg} if:
	\begin{enumerate}
		\item\label{solitem1} $u$ is a subsolution of the Hamilton-Jacobi equation; that is, 
		\begin{equation*}
			\begin{cases}
				-u_t+{H(x,\nabla u)}\le f(t,x,m) ,\\
				u(T,x)=\psi(x) ,
			\end{cases}
		\end{equation*}
		in the sense of distributions (see \eqref{eq:distribution} with $\alpha=f(t,x,m)$), and $u(T,x)=\psi(x)$ in the trace sense;
		\item\label{solitem2} almost everywhere in the set $\{m>0\}$, $u$ satisfies 
		\begin{equation*}
			-u_t^{ac}+{H(x,\nabla u)}=f(t,x,m) ,
		\end{equation*}
		where $u_t^{ac}$ is the absolutely continuous part of the measure $u_t$;
		\item\label{solitem3} $m$ satisfies the Fokker-Planck equation
		\begin{equation*}
			\begin{cases}
				m_t-\mathrm{div}(m{H_p(x,\nabla u)})=0 ,\\
				m(0)=m_0 ,\\
				m{H_p(x,\nabla u)}\cdot\nu(x)_{|x\in\partial\Omega}=j(t,x) ,
			\end{cases}
		\end{equation*}
		in the sense of distribution \eqref{either} with $w=-m{H_p(x,\nabla u)}$;
		\item\label{solitem4} $(1+m)|\nabla u|^r\in L^1(Q_T)$;
		\item\label{solitem5} The following equality holds 
		with $w=-m{H_p(x,\nabla u)}$ and $\alpha=f(t,x,m)$:
		\begin{align}\label{dual_eq}
			&\into u(0) m_0 dx+\intb ju dxdt
			\\\notag&\qquad =\into m(T,x)\psi(x)dx+\intif\left[m{L\left(x,\frac wm\right)}+m\alpha\right] dxdt .
		\end{align}
	\end{enumerate} 
\end{definition}
Note that the condition in Item~\ref{solitem4} implies that ${H(x,\nabla u)}$ is integrable enough 
so that it makes sense to say that the Hamilton-Jacobi equation holds in the sense of distributions
in Item~\ref{solitem1}. Moreover, because  
$$
m{L\left(x,\frac wm\right)}=m{H^*(x,H_p(x,\nabla u))}=m\left({H_p(x,\nabla u)}\nabla u -v c{H(x,\nabla u)}\right)
$$
is bounded by $Cm(1+|\nabla u|^r)$, Item~\ref{solitem5} ensures that
the last integral in Item~\ref{solitem5} is well-defined. Also, we note that classical solutions of the MFG system are solutions in the sense of Definition \ref{defmfg}.

This solution definition is standard in the literature of first-order MFG systems; see  \cite{cgbt,Card1order,CardGrab, Tono2019}. 
The lack of regularity due to the absence of a Laplacian (or an elliptic second-order term) implies that $u$ may not be a solution of the HJB equation in the whole set $\Omega$. The condition in Item~\ref{solitem2} complements the subsolution requirement for $u$ by requiring equality almost everywhere in the set $\{m>0\}$, replacing $u_t$ with its absolutely continuous part.
Because $u$ is not a solution in $\Omega$, the uniqueness of solutions for the Problem \ref{problem_mfg} could be lost in the set $\{m=0\}$. This is reflected in the statement of Theorem \ref{thm:main}, where we assert only partial uniqueness for $u$. 
For the Fokker-Planck equation, we can require a distributional solution, as in Item~\ref{solitem3}, in duality with the Hamilton-Jacobi one. Finally, the condition in Item~\ref{solitem5} can be formally obtained using $u$ as a test function for the equation of $m$.

%% file: Variational.tex
In this section, we further examine the variational Problems \ref{var1} and \ref{var2}. First, we prove that the solutions of Problem \ref{var1} also solve Problem \ref{problem_mfg} (see Proposition \ref{pro_1promfg}). Next, in Subsection \ref{efp3}, 
we establish preliminary estimates for Problem \ref{var2}. Finally,
we reformulate Problem \ref{var2} to obtain an equivalent formulation that is more suitable for our analysis (see Remark \ref{remark35}).

\subsection{Optimality conditions}

We begin by examining  the Euler-Lagrange equation associated with the functional $\mathcal I(\cdot)$ in Problem \ref{var1}
and  its relation with with \eqref{mfg}.
\begin{pro}\label{pro_1promfg}
Suppose that Assumption \ref{hp} holds, and let $u$ solve Problem \ref{var1}.
Then, there exists $m\in \mathcal C(Q_T)$ such that $(u,m)$ solves 
Problem \ref{problem_mfg}.
\end{pro}
\begin{proof}
Because $u$ solves Problem \ref{var1},  $u\in\mathscr I$ and it is a minimizer of  $\mathcal I(\cdot)$. 
For $\eps\in\R$, consider
the function $z:=u+\eps v\in\mathscr I$, with $v\in\mathcal C^\infty(Q_T)$ and $v(T)=0$. 
Because $u$ is a minimizer, $\mathcal I(u)\le \mathcal I(z)$ for all $z\in\mathscr I$.
Accordingly, the function $\varepsilon \mapsto \mathcal I(u+\eps v)$ has a minimum in $\eps=0$. Thus, we have
\begin{align*}
\mathcal I(u+\eps v)&=\intif F^*(t,x,-u_t-\eps v_t+{H(x,\nabla u+\eps \nabla v)})dxdt\\&-\into \big(u(0,x)+\eps v(0,x)\big)m_0(x)\,dx-\int_0^T\int_{\partial\Omega}j(u+\eps v)dx dt.
\end{align*}
Differentiating with respect to $\eps$ and evaluating at $\eps=0$, we find
\begin{align*}
&\intif F^*_s(t,x,-u_t+{H(x,\nabla u)})(-v_t+{H_p(x,\nabla u)}\cdot \nabla v)dxdt\\
-&\into v(0,x)m_0(x)\,dx-\int_0^T\int_{\partial\Omega}jvdx dt=0\,.
\end{align*}
We now define
\begin{equation}\label{eq:m}
m(t,x):=F_s^*(t,x,-u_t+{H(x,\nabla u)})\,.
\end{equation}
Using \eqref{hp:CC}, we apply $f$ to both sides of \eqref{eq:m}.  Because $u\in\mathscr I$, we get
$$
\begin{cases}
-u_t+{H(x,\nabla u)}=f(t,x,m), & (t,x)\in (0,T)\times\Omega\,,\\
u(T,x)=\psi(x), & x\in\Omega\,.
\end{cases}
$$
This implies
$$
\intif m(t,x)(-v_t+{H_p(x,\nabla u)}\cdot \nabla v)\,dxdt=\into v(0,x)\,m_0(x)\,dx+\int_0^T\int_{\partial\Omega}jv\,dx dt
$$
for all $v\in\mathcal C^\infty(Q_T)$ with $v(T,\cdot)=0$; that is, $m$ is a distributional solution to
$$
\begin{cases}
m_t-\mathrm{div}(m {H_p(x,\nabla u)})=0\,,\\
m(0)=m_0\,,\\
m{H_p(x,\nabla u)}\cdot\nu_{|x\in\partial\Omega}=j\,.
\end{cases}
$$
Accordingly, $(u,m)$ solves \eqref{mfg}.
\end{proof}

Unfortunately, 
the existence of minimizers for $\mathcal I(\cdot)$ in $\mathscr I$ is not guaranteed due to 
 insufficient regularity in the first-order case and the lack of coercivity.
To address this matter, in Section \ref{sec5}, we introduce a relaxation of Problem \ref{var1} for which the existence of a minimizer can be proven.

\subsection{Estimates for  Problem \ref{var2}}
\label{efp3}

Now, we focus on Problem \ref{var2}. Using \eqref{either} with $\xi=T-t$, we derive
the following  $L^1$ bound on the density, $m$, of a solution of Problem \ref{var2}:
\begin{equation}\label{stimamL1}
\|m\|_{L^1(Q_T)}\le T\left(\|m_0\|_{L^1}+\|j\|_{L^1}\right).
\end{equation}

The following propositions provide additional integrability estimates for solutions of Problem \ref{var2}.
\begin{pro}\label{prop_regM}
Consider the setting of Problem \ref{var2}.
Let $(m,w)\in\mathscr M$, and suppose that Assumption \ref{hp} holds.

If ${\mathcal M(m,w)<+\infty}$, we have $m\in L^q(Q_T)$ and $w\in L^\beta(Q_T)$, with $\beta=\frac{qr}{qr-q+1}$.
\end{pro}
\begin{proof}
If $\mathcal M(m,w)<+\infty$, we have $\min\limits_{(m,w)\in\mathscr M}\mathcal M(m,w)\le C$ for some constant $C>0$, i.e.
\begin{equation}\label{eq:M}
\begin{split}
\intif \Big[m\,{L\left(x,\frac wm\right)}&+F(t,x,m)+w\cdot\nabla \psi(x)\Big]\,dxdt\\&+\into \psi(x)m_0(x)\,dx+\int_0^T\int_{\partial\Omega}\psi(x)\,j(t,x)\,dx dt\le C
\end{split}
\end{equation}

Thanks to Assumption \ref{hp}, we can estimate the integrals in the left-hand side: the $L^\infty$ bounds for $\psi$ and $j$ imply
\[
\into \psi(x)m_0(x)\,dx\ge -\norminf{\psi},
\]
and
\[
\intb \psi jdx dt\ge-C\norminf{\psi}\norminf{j},
\]
whereas the growth assumption on $F$ and Remark \ref{rem21} imply
\[\intif F(t,x,m)dxdt\ge C^{-1}\|m\|_{L^q}^q - C\,,
\]
and
\[
 \intif m{L\left(x,\frac wm\right)}dxdt\ge C^{-1}\iint_{\{m>0\}} \frac{|w|^{r'}}{m^{r'-1}}dxdt -C\,.
\]
Finally, for the last integral in $\mathcal M(m,w)$, we use a generalized Young's inequality with exponents $r'$ and $r$ to get
$$
\intif w\cdot\nabla \psi(x)dxdt\ge -\norminf{\nabla\psi}\iint_{\{m>0\}}\frac{|w|}{m^{\frac1r}}m^{\frac1r}dxdt\ge -\frac{C^{-1}}2\iint_{\{m>0\}} \frac{|w|^{r'}}{m^{r'-1}}\,dxdt -C\,,
$$
where in the last step, we use \eqref{stimamL1}. Coming back to \eqref{eq:M} we find
$$
\frac1{2C}\iint_{\{m>0\}} \frac{|w|^{r'}}{m^{r'-1}}\,dxdt+C^{-1}\norm{m}_{L^q}^q\le C\big(1+\norminf{\psi}+\norminf\psi\norminf j\big)\,.
$$

Hence, $m\in L^q(Q_T)$ and $\frac{|w|^{r'}}{m^{r'-1}}\in L^1(Q_T)$.

From here, we derive the following bound on $w$. If $\gamma, s>1$, $\alpha>0$, we have, by H\"older's inequality,
\[
\intif |w|^{\gamma}dxdt\le \left(\intif \frac{|w|^{\gamma s}}{m^{\alpha s}}dxdt\right)^{\frac 1s}\left(\intif m^{\alpha s'}dxdt\right)^{\frac 1{s'}}.
\]
The right-hand side is bounded if $\gamma s=r'$, $\alpha s=r'-1$, $\alpha s'=q$; that is,  with 
$$
\gamma=\frac{r'q}{r'+q-1}=\frac{qr}{qr-q+1}\,,\qquad s=\frac{r'+q-1}q\,,\qquad \alpha=\frac{(r'-1)q}{r'+q-1}\,.
$$
Then, $w\in L^{\frac{qr}{qr-q+1}}$, which concludes the proof.
\end{proof}

In the sequel, we use the following notation. Let  $\bo{\rho}=(m,w)$ and define
\begin{equation}\label{bold}
\bo{\mathrm{div}}(\cdot)=\mathrm{div}_{(t,x)}(\cdot)\,,\qquad \bo D=(\partial_t,\nabla)\,.
\end{equation}
The first equation in \eqref{eq:fp} can be written as
$$
\bo{\mathrm{div}}(\bo \rho)=0\,.
$$

Next, we establish a useful property for the function $m$ by proving the following proposition.
\begin{pro}\label{pro33}
	Assume $\theta>1$ and  let $(m,w)\in L^\theta(Q_T;\R^{N+1})$ solve
Problem \ref{var2}.
Then, for all $t\in(0,T]$, we have  $m(t)\in \left(\mathcal C^\gamma(\Omega)\right)'$, for any $\gamma\in (\theta^{-1},1)$. 
\end{pro}
\begin{proof}
Let $t\in(0,T]$. Since the pair $(m,w)$ satisfies $\bo{\mathrm{div}}(\bo \rho)=0$, thanks to \emph{Lemma II.1.2.2} of \cite{Sohr}, there exists a bounded linear functional $\tr(\bo \rho)\in \big(W^{\frac1\theta,\theta'}(\partial_PQ_t)\big)'$ which extends the classical trace at the boundary. This means that, when $\bo \rho\in\mathcal C(Q_t)$, for all $\phi\in W^{\frac1\theta,\theta'}(\partial_PQ_t)$
$$
\langle\tr(\bo \rho),\phi\rangle=\into m(t)\phi(t,x)dx-\into m_0(x)\phi(0,x)dx+\int_0^t\int_{\partial\Omega}w(s,x)\cdot\nu(x)\phi(s,x)dxds\,.
$$
Since $w(s,x)\cdot\nu(x)=-j(s,x)$ when $x\in\partial\Omega$ and $m_0\in\mathcal C^1(\Omega)$, the last two integrals in the right-hand side are well-defined, even if $\bo \rho$ is not continuous.

Now, take $\phi\in\mathcal C^\gamma(\Omega)$. We can extend $\phi$ to $\partial_PQ_t$, by defining $\tilde\phi(t,x)=\phi(x)$ for all $(t,x)\in\partial_PQ_t$. Then $\tilde\phi\in\mathcal C^\gamma(\Omega)\subset W^{\frac1\theta,\theta'}(\Omega)$, and we can define
\begin{equation}\label{defmt}
\langle m(t),\phi\rangle=\langle\tr(\bo \rho),\tilde\phi\rangle+\into \phi(x)m_0(x)\,dx+\int_0^t\int_{\partial\Omega}j(s,x)\phi(x)dx ds\,.
\end{equation}
\end{proof}

Henceforth, we use the notation $\into \phi(x)m(t,dx)$ to indicate the quantity $\langle m(t),\phi\rangle$. An immediate consequence of the previous Proposition is the following Corollary.
\begin{cor}
Let $\theta>1$ and let $\bo \rho\in L^\theta(Q_T;\R^{N+1})$ solve \eqref{eq:fp}. Then, for all $t\in(0,T]$ and $\xi\in W^{1,\infty}(Q_T)$, we have
\begin{align}
\inti(m\xi_t+w\cdot\!\nabla\xi)\,dxds+\into\xi(0,x)m_0(x)\,dx+\int_0^t\int_{\partial\Omega} j\xi\, dx ds=\into\xi(t,x)\,m(t,dx)\,,\label{generalformula}\\
\intf(m\xi_t+w\cdot\!\nabla\xi)\,dxds+\into\xi(t,x)\,m(t,dx)+\int_t^T\int_{\partial\Omega} j\xi\, dx ds=\into\xi(T,x)\,m(T,dx)\,.\notag
\end{align}
\end{cor}
\begin{proof}
Let $\xi\in W^{1,\infty}(Q_T)$. Then, the function
$$
\varphi(s,x):=\begin{cases}
\xi(s,x)-\xi(t,x) &\text{if }0\le s\le t\,,\\
0 &\text{if }t<s\le T
\end{cases}
$$
satisfies $\varphi\in W^{1,\infty}(Q_T)$ and $\varphi(T,\cdot)=0$. Then, by \eqref{either}, we have
\begin{align*}
&\inti (m\xi_t+w\cdot\nabla\xi)dxds+\into \xi(0,x)m_0(x)\,dx+\int_0^t\int_{\partial\Omega} j(t,x)\xi(s,x)dx ds\\
&\quad=\inti w\cdot\nabla\xi(t,x)dxds+\into\xi(t,x)m(t,dx)+\int_0^t\int_{\partial\Omega}j(t,x)\xi(t,x)dx ds\,.
\end{align*}
The first integral on the left-hand side can be rewritten as
$$
\inti \bo \rho\cdot \bo D\xi(s,x)dxds=-\inti\bo{\mathrm{div}}(\bo \rho)\xi(t,x)dxds+\langle\tr(\bo \rho),\xi(t,\cdot)\rangle=\langle\tr(\bo \rho),\xi(t,\cdot)\rangle\,.
$$
Since $\xi\in W^{1,\infty}(Q_T)$, we have from the Sobolev embedding $\xi(t)\in\mathcal C^\gamma(\Omega)$ for all $\gamma<1$. Then, we can use \eqref{defmt} with $\phi(\cdot)=\xi(t,\cdot)$, which implies directly the first equation of \eqref{generalformula}. Computing it for $t=T$, for a generic $t\in(0,T]$ and subtracting the two equations, we obtain the second one. This concludes the proof. 
\end{proof}

Reformulating Problem \ref{var2} as presented below simplifies the functional $\mathcal M(m,w)$ by directly integrating the terminal condition, thus enabling a more direct application of variational techniques. This reformulation is used in the proof of Proposition \ref{reverse}. 

\begin{remark}
\label{remark35}
An equivalent formulation for Problem \ref{var2} use the following alternative representation,
\begin{equation}\label{eq:seconda}
\mathcal M(m,w)=\into \psi(x)m(T,dx)+\intif \left[m{L\left(x,\frac wm\right)}+F(t,x,m)\right]dxdt\,,
\end{equation}
which holds provided $m(T)$ is well defined as an element of $(\mathcal C^1(\Omega))'$.

The first integral of \eqref{eq:seconda} can be written as
\[
\into \psi(x)m(T,dx)=\langle\tr(\bo \rho),\psi\rangle+\intb\psi j\,dx dt+\into\, \psi m_0(x)\,dx\,.
\]
The last two integrals appear in the definition of $\mathcal M$ \eqref{def:M}. Concerning the first one, integrating by parts, we get
$$
0=\intif\bo{\mathrm{div}}(\bo \rho)\psi(x)dxdt=-\intif\bo \rho\cdot\bo{D}\psi(x)dxdt+\langle\tr(\bo \rho),\psi\rangle\,,
$$
which implies
$$
\langle\tr(\bo \rho),\psi\rangle=\intif w(t,x)\cdot\nabla \psi(x)dxdt\,,
$$
which proves the equivalence of \eqref{def:M} and \eqref{eq:seconda}.

We observe that $\bo \rho\in L^\theta(Q_T;\R^{N+1})$, with $\theta=\beta>1$ given by Proposition \ref{prop_regM} (observe that we always have $q>\beta$). Hence, $m(T)$ is well-defined thanks to Proposition \ref{pro33}, and Problem \ref{var2} can be reformulated by replacing \eqref{def:M} with \eqref{eq:seconda}.
\end{remark}

%% file: Duality.tex
Building on \cite[Lemma 2]{Card1order}, we explore the duality between Problems \ref{var1} and \ref{var2}.
The main result in this section  is the following theorem, which establishes the existence of a solution for Problem \ref{var2}.
\begin{teo}\label{duality}
Consider the setting of Problems \ref{var1} and \ref{var2}. 
Suppose Assumption \ref{hp} holds. Then there exists a unique minimizer $(m,w)$ of $\mathcal M$, with $m\in L^q(Q_T)$ and $w\in L^{\frac{qr}{qr-q
+1}}(Q_T;\R^N)$. Additionally, 
\begin{equation}\label{eq:duality}
\inf\limits_{v\in\mathscr I}\mathcal I(v)=-\min\limits_{(m,w)\in\mathscr M}\mathcal M(m,w) .
\end{equation}
\end{teo}
\begin{proof}
We adapt the proof techniques from \cite[Lemma 2]{Card1order},  where a proof was given for $\Omega=\mathbb T^d$, 
to accommodate Neumann boundary conditions. 

Consider the convex set $\mathscr I\subset\mathcal C^1(Q_T)$ as in \eqref{eq:Iset}, and define the lower semicontinuous function $\Upsilon_{\mathscr I}:\mathcal C^1(Q_T)\to\R\cup\{+\infty\}$ as
\[
\Upsilon_{\mathscr I}(v)=\begin{cases}
	0\quad&\text{if } v\in\mathscr I\\
	+\infty &\text{otherwise}.
\end{cases}
\]
We use $\Upsilon_{\mathscr I}$ as a penalty to enforce the terminal condition.

For $v\in\mathcal C^1(Q_T)$, let
$$
\mathcal F(v)=-\into v(0,x)m_0(x) dx-\intb jv dx dt+\Upsilon_{\mathscr I}(v) .
$$
Let
$$
E:=\mathcal C(Q_T)\times \mathcal C(Q_T;\R^N) .
$$
For $(a,b)\in E$, we set
$$
\mathcal G(a,b)=\intif F^*\big(t,x,-a+{H(x,b)}\big) dxdt .
$$
The infimum $\inf\limits_{v\in\mathscr I}\mathcal I(v)$ can be expressed as
\begin{equation}\label{P}
\inf\limits_{v\in\mathcal C^1(Q_T)}\left\{\mathcal F(v)+\mathcal G(\bo{D}v)\right\} ,
\end{equation}
where $\bo{D}:\mathcal C^1(Q_T)\to E$ is defined in \eqref{bold}. Following \cite[Chapter III]{Ekeland}, we define the functional
$$
\Phi:\mathcal C^1(Q_T)\times E\to \R ,\qquad \Phi(v,\bo p)=\mathcal F(v)+\mathcal G(\bo Dv-\bo p) .
$$

Accordingly, problem \eqref{P} takes the form
$$
\inf\limits_{v\in\mathcal C^1(Q_T)}\Phi(v,0) ,
$$
and the dual problem is (see \cite[page 47]{Ekeland})
\begin{equation}\label{Pstar}
	\max\limits_{\bo p^*\in E^*}\{-\Phi^*(0,\bo p^*)\} ,
\end{equation}
where $E^*$ is the topological dual of $E$, whereas $\Phi^*$ stands for the convex conjugate of $\Phi$. 

Due to Assumption \ref{hp}, $H$, $m_0$, and $j$ are continuous. Moreover, thanks to Remark \ref{rem21}, $F^*$ is also continuous. Since $\Upsilon_{\mathscr I}$ is lower semicontinuous, we have that both $\mathcal F$ and $\mathcal G$ are lower semicontinuous functions, respectively on $\mathcal C^1(Q_T)$ and $E$.

Thanks to Remark \ref{rem21} and Remark \ref{rem23}, $F^*$ is convex and non-decreasing and $H$ is convex. Thus, $\mathcal F$ and $\mathcal G$ are convex. Additionally,
the function $v(t,x)=\psi(x)$ satisfies $\mathcal F(v)+\mathcal G(\bo Dv)<+\infty$ and $\mathcal G$ is continuous at $\bo Dv$. Then, we can apply \cite[Theorem III.4.1]{Ekeland}, which states the equality of the quantities in \eqref{P} and \eqref{Pstar}, i.e.,
\[
\inf\limits_{v\in\mathscr I}\mathcal I(v)=\max\limits_{\bo p^*\in E^*}\{-\Phi^*(0,\bo p^*)\} ,
\]
where the convex conjugate is taken in both variables.

Computing the conjugate in \eqref{Pstar}, and defining $\bo q=\bo Dv-\bo p$, we find
\begin{align*}
\Phi^*(0,\bo p^*)&=\sup\limits_{(v,\bo p)\in\mathcal C^1(Q_T)\times E}\big\{\langle \bo p^*,\bo p\rangle -\mathcal F(v)-\mathcal G(\bo Dv-\bo p)\big\}\\
&=\sup\limits_{(v,\bo q)\in\mathcal C^1(Q_T)\times E}\big\{\langle \bo p^*,\bo Dv\rangle+\langle -\bo p^*,\bo q\rangle -\mathcal F(v)-\mathcal G(\bo q)\big\}\\
&=\!\!\!\sup\limits_{v\in\mathcal C^1(Q_T)}\!\!\big\{\langle \bo D^*\bo p^*,v\rangle-\mathcal F(v)\big\}+\sup\limits_{\bo q\in E}\big\{\langle -\bo p^*,\bo q\rangle -\mathcal G(\bo q)\big\}=\mathcal F^*(\bo D^*\bo p^*)+\mathcal G^*(-\bo p^*),
\end{align*}
where $\bo D^*(\cdot):E^*\to \mathcal C^1(Q_T)^*$ denotes the adjoint operator of $\bo D$, i.e., for $\bo p^*\in E^*$ and $v\in\mathcal C^1(\Omega)$,
$$
\bo D^*(\bo p^*)(v)=\bo p^*(\bo D v) .
$$
Note that each element of $E^*$ can be represented as a couple $(m,w)\in\mathcal C(Q_T)^*\times\mathcal C(Q_T;\R^N)^*$.
Hence, we have
\begin{equation}\label{eq:rock}
\begin{split}
\inf\limits_{v\in\mathscr I}\mathcal I(v)&=\max\limits_{\substack{m\in\mathcal C(Q_T)^*\\w\in\mathcal C(Q_T;\R^N)^*}}\{-\mathcal F^*(\bo D^*(m,w))-\mathcal G^*(-m,-w)\}\\
&=-\min\limits_{\substack{m\in\mathcal C(Q_T)^*\\w\in\mathcal C(Q_T;\R^N)^*}}\{\mathcal F^*(\bo D^*(m,w))+\mathcal G^*(-m,-w)\} .
\end{split}
\end{equation}

We recall that $\mathcal C(Q_T)^*$ is the set of signed Radon measures over $Q_T$, and the same applies to $\mathcal C(Q_T;\R^N)^*$. For the first term, we have
\begin{align*}
\begin{array}{c}
\displaystyle\mathcal F^*(\bo D^*(m,w))=\sup\limits_{v\in\mathcal C^1(Q_T)}\left\{ \intif v_t dm(t,x)+\intif \nabla v\cdot dw(t,x)-\mathcal F(v)\right\}\\
\displaystyle=\sup\limits_{v\in\mathscr I}\left\{ \intif v_t dm(t,x)+\intif \nabla v\cdot dw(t,x)+\into v(0,x) m_0(x) dx+\intb{jv dx dt}\right\}.
\end{array}
\end{align*}

To establish that $\mathcal F^*(\bo D^*(m,w))$ is finite if and  only if $(m,w)$ solves \eqref{eq:fp}, let  $J(v)$ be the functional to be maximized on the right-hand side of the 
previous identity. If $(m,w)$ does not solve \eqref{eq:fp}, then \eqref{either} fails, i.e. there exists $\bar\xi\in W^{1,\infty}(Q_T)$ and $\alpha\neq0$ with $\bar\xi(T,\cdot)=0$ such that
$$
\intif (m\bar\xi_t+w\cdot\nabla\bar\xi) dxdt+\into \bar\xi(0,x) m_0(x) dx+\intb j\bar\xi dx dt=\alpha .
$$
By replacing $\bar\xi$ with $-\bar\xi$ when necessary, we can assume that $\alpha>0$ without loss of generality.
We choose $v_k=\psi+k\bar\xi$, where $k>0$. Then $v_k\in\mathscr I$. Integrating by parts, we obtain
\begin{equation*}
\begin{split}
J(v_k)=\intif\nabla \psi\cdot dw(t,x)+\into \psi(x) m_0(x) dx+\intb j\psi dx dt+k\alpha .
\end{split}
\end{equation*}
As $k\to+\infty$, we obtain
$$
\lim\limits_{k\to+\infty}J(v_k)=+\infty\implies \mathcal F^*(\bo D^*(m,w))=+\infty .
$$
If $(m,w)$ satisfies \eqref{eq:fp}, then from \eqref{either} we have $J(\xi)=0$ for every test function $\xi\in W^{1,\infty}$ with $\xi(T)=0$. Since $\psi-v$ can be used as a test function $\forall v\in\mathscr I$, using the linearity of $J$, we have
$$
J(\psi-v)=0\implies J(v)=J(\psi)\implies \mathcal F^*(\bo D^*(m,w))=J(\psi) .
$$
In conclusion, we have
\begin{align*}
&\mathcal F^*(\bo D^*(m,w))=\\
&\begin{cases}
\displaystyle\intif\nabla \psi\cdot dw(t,x)+\into \psi(x) m_0(x) dx+\intb j\psi dx dt\ &\hbox{if $(m,w)$ solves \eqref{eq:fp}} ,\\
\displaystyle +\infty &\text{otherwise} .
\end{cases}
\end{align*}
Similar calculations for $\mathcal G^*$, based on \cite[Lemma 2]{Card1order}, yield
\begin{equation*}
\mathcal G^*(-m,-w)=\left\{\begin{array}{lr}
\displaystyle\intif\left[F(t,x,m)+m {L\left(x,\frac wm\right)}\right]dxdt , &\text{if }m\ge0\quad (m,w)\in L^1,\\
\displaystyle +\infty &\text{otherwise} .
\end{array}\right.
\end{equation*}
Hence, $\mathcal F^*(\bo D^*(m,w))$ and $\mathcal G^*(-m,-w)$ are both finite if and only if $(m,w)\in\mathscr M$. Substituting the representation formulas of $\mathcal F$ and $\mathcal G$ in \eqref{eq:rock}, we finally get
\[
\inf\limits_{v\in\mathscr I}\mathcal I(v)=-\min\limits_{(m,w)\in\mathscr M}\mathcal M(m,w) .
\]

To prove the uniqueness of a minimizer for $\mathcal M$, assume by contradiction that there are two distinct minimizers of $\mathcal M$, $(m^0,w^0)\in\mathscr M$ and $(m^1,w^1)\in\mathscr M$. Let $a$ be the minimum value attained. Then, for all $s\in[0,1]$,  $(m^s,w^s):=\big((1-s)m^0+s m^1,(1-s)w^0+s w^1\big)\in\mathscr M$.  By the strict convexity of $F$ and $H^*$ (the last one follows from Lemma \ref{lem:strictHstar}), we have
\begin{equation}\label{reductioababsurdum}
\begin{split}
\mathcal M(m^s,w^s)&<\into \psi(x)m_0(x) dx+\intb \psi j dx dt+\into w^s\cdot\nabla\psi dxdt\\
&+(1-s)\intif F(t,x,m^0) dxdt+s\intif F(t,x,m^1) dxdt\\
&+\intif m^s{L\left(x,\frac{w^s}{m^s}\right)} dxdt .
\end{split}
\end{equation}
To estimate the last integral, we use the convexity of $L$ {in the last variable}. Let $\lambda=\frac{(1-s)m^0}{m^s}$. Then, we have $1-\lambda=\frac{sm^1}{m^s}$. Hence, we get
\begin{equation*}
\begin{split}
m^s {L\left(x,\frac{w^s}{m^s}\right)}=m^s{L\left(x,\lambda\left(\frac{w^0}{m^0}\right)+(1-\lambda)\left(\frac{w^1}{m^1}\right)\right)}\\
< (1-s)m^0 {L\left(x,\frac{w^0}{m^0}\right)}+s m^1 {L\left(x,\frac{w^1}{m^1}\right)}.
\end{split}
\end{equation*}
Using this relation in \eqref{reductioababsurdum}, we obtain
\[
\mathcal M(m^s,w^s)<(1-s)\mathcal M(m^0,w^0)+s\mathcal M(m^1,w^1)=(1-s)a+sa=a ,
\]
which is not possible because $a$ is the minimum of $\mathcal M$. This implies $m^0=m^1$ and $\frac{w^0}{m^0}=\frac{w^1}{m^1}$, i.e. $w^0=w^1$ for $m^0\neq 0$, $m^1\neq 0$. Since $w^0=0$ a.e. in $\{m^0>0\}$, we must have $w^1=0$, and the uniqueness is proved.

Therefore, we have established a duality relationship between the primal problem of minimizing $\mathcal I(v)$ over $v\in\mathscr I$ and the dual problem of minimizing $\mathcal M(m,w)$ over $(m,w)\in\mathscr M$. We have also proved the existence and uniqueness of a minimizer $(m,w)$ for the functional $\mathcal M$ under Assumption \ref{hp}. 
To conclude, we need to check the integrability conditions on the minimizer. 
Since  $\min\limits_{(m,w)\in\mathscr M}\mathcal M(m,w)<+\infty$, we have $m\in L^p(Q_T)$ and $w\in L^\frac{pr}{pr-p+1}(Q_T)$ thanks to Proposition \ref{prop_regM}. 
\end{proof}

%% file: Relaxation.tex
Because Problem \ref{var1} may not have a solution, we introduce the following relaxed problem by enlarging the set $\mathscr I$.
\begin{problem}\label{var3}
Consider the set $\mathscr K$ of all pairs $(v,\alpha)\in BV(Q_T)\times L^{q'}(Q_T)$ such that  ${\nabla v\in L^r(Q_T)}$ and the pair $(v,\alpha)$ satisfies
\begin{equation}\label{eq:fpgen}
\begin{cases}
-v_t+{H(x,\nabla v)}\le\alpha ,\\
v(T,x)\le\psi(x) ,
\end{cases}
\end{equation}
where the prior inequalities are interpreted as follows. 
The first inequality means that 
for each $\phi\in\mathcal C_c^{\infty}((0,T]\times\overline\Omega)$, $\phi\ge0$, we have
\begin{equation}\label{eq:distribution}
\intif \Big(v\phi_t+{H(x,\nabla v)}\phi\Big) dxdt\le\intif\alpha\phi dxdt+\into \psi(x)\phi(T,x) dx .
\end{equation}
The second inequality in \eqref{eq:fpgen} is interpreted in the trace sense since
$v\in BV(Q_T)$ and, therefore, the trace $v_{|\partial\Omega}$ is well defined.

Find $(v,\alpha)\in\mathscr K$ minimizing the functional
\[
\mathcal K(v,\alpha):=\intif F^*(t,x,\alpha) dxdt-\into v(0,x)dm_0(x)-\intb jv dx dt .
\]
\end{problem}
In the definition  of $\mathscr K$, we can replace $Q_T$
with a domain $[a,b]\times A$, where $A\subset\R^N$. In that case, we denote that set by $\mathscr K([a,b]\times A)$.

First, we prove that Problem \ref{var3} is a relaxation of Problem \ref{var1}, since both have the same infimum value. Hence, the duality with Problem \ref{var2} is preserved.
\begin{teo}\label{thm:duality}
Consider the setting of Problems \ref{var1}-\ref{var3}. 
Suppose Assumptions \ref{hp} and \ref{hp2} hold. Then, we have
\begin{equation}\label{FinalDuality}
\inf\limits_{(v,\alpha)\in\mathscr K}\mathcal K(v,\alpha)=\inf\limits_{v\in\mathscr I}\mathcal I(v) =-\min\limits_{(m,w)\in\mathscr M}\mathcal M(m,w) .
\end{equation}
\end{teo}

To prove Theorem \ref{thm:duality}, we first establish two technical lemmas that will aid in handling the relaxed problem.
To regularize the pairs $(v,\alpha)\in\mathscr K$, the first lemma provides a smoothing technique that preserves \eqref{eq:fpgen}.

\begin{lem}\label{lem:regularization}
Consider the setting of Problems \ref{var1} and \ref{var3}. 	
Suppose Assumptions \ref{hp} and \ref{hp2} hold, and let $(v,\alpha)\in\mathscr K$. Then, for all $\eps>0$, there exists $(v^\eps,\alpha^\eps)\in\mathscr I\times\mathcal C^\infty(Q_T)$, which satisfies \eqref{eq:fpgen} in the classical sense, and such that $(v^\eps,\alpha^\eps,\nabla v^\eps)\to(v,\alpha,\nabla v)$ in $L^1(Q_T)\times L^{q'}(Q_T)\times L^r(Q_T)$, as $\eps\to0$.
\end{lem}
\begin{proof} The idea is to use an approximation by convolution with smooth mollifiers. Since
by convolution 	
	 we can only expect a convergence in the interior, we need to enlarge the domain $Q_T$ to ensure the convergence in the whole $Q_T$.

We split the proof into three steps: first, we discuss how to extend the domain, then we address the extensions by reflection, and finally, we use convolution with a smooth mollifier to prove the lemma. 

\emph{Step 1. Extension of the domain.} Let $(v,\alpha)\in\mathscr K$. Recall that, thanks to Item \ref{item5} of Assumption $2$, $\Omega$ is a rectangular domain.

First, we extend $(v,\alpha)$ to a larger domain ${Q^\eps_T\supset Q_T}$. To accomplish this, we extend $\Omega=\prod\limits_{i=1}^N(a_i,b_i)$ by reflecting each interval $(a_i,b_i)$ with respect to $a_i$ and $b_i$. New intervals are  $(2a_i-b_i,2b_i-a_i)$, and the extension of $\Omega$ is $U=\prod\limits_{i=1}^N(2a_i-b_i,2b_i-a_i)$.

Next, we define auxiliary functions $\vfi_1$, $\vfi_2$ and $\vfi_3$ as follows.
\begin{equation}\label{vfi}
\vfi_1^i(x_i)=2a_i-x_i ,\qquad \vfi_2^i(x_i)=x_i ,\qquad \vfi_3^i(x_i)=2b_i-x_i ,\qquad x_i\in[a_i,b_i] .
\end{equation}

Let $\bo j=(j_1,\dots,j_N)$ be a multi-index in $\{1,2,3\}^N$, and
\begin{equation}\label{Phij}
\varPhi_{\bo j}(x)=(\vfi_{j_1}^1(x),\dots,\vfi_{j_n}^n(x)) ,\qquad x\in\Omega .
\end{equation}
Then, we can partition $U$ as
$$
U=\bigcup\limits_{\bo j\in\{1,2,3\}^N}\Omega_{\bo j} ,\qquad\text{with }\Omega_{\bo j}=\big\{\varPhi_{\bo j}(x)\ |\ x\in\Omega\big\} .
$$
In the preceding decomposition, if $j_i=2$ there is no reflection of the $i$-th interval $[a_i,b_i]$, if $j_i=1$ the reflection is on the left side, and if $j_i=3$ it is on the right side. From \eqref{vfi}, we observe that
\begin{equation}\label{propvfi}
\begin{split}
\big(\vfi^i_j\big)^{-1}(x_i)=\vfi^i_j(x_i) ,\qquad (\vfi^i_j)'(x_i)=(-1)^j ,\\ 
\mathrm{Jac} \varPhi_{\bo j}(x)=\mathrm{Diag}[(-1)^{j_1},\dots,(-1)^{j_N}] ,\qquad j=1,2,3 .
\end{split}
\end{equation}
Then, we define the extension of $Q_T^\eps$ of $Q_T$ to be
\begin{equation}\label{QTeps}
Q_T^\eps=[-2\eps,T+2\eps]\times U .
\end{equation}

\emph{Step 2. Extension of the functions.} In this step, we consider a pair $(v,\alpha)\in\mathscr K$ and extend it to $Q_T^\eps$, modifying slightly the time dependence. 

For $\lambda\in\R$, which we choose later, let $(z^\eps,\beta^\eps)\in (BV\times L^{q'})([-2\eps,T+2\eps]\times\Omega)$ be defined in the following way:
\begin{equation*}
z^\eps(t,x)=\left\{
\begin{array}{lll}
v\big(t+2\eps,x\big) \quad& \text{if }\hspace{-0.2cm}&(t,x)\in[-2\eps,T-2\eps]\times\Omega ,\\
\psi(x)+\lambda(T-2\eps-t) \quad& \text{if }\hspace{-0.2cm}&(t,x)\in[T-2\eps,T+2\eps]\times\Omega ,
\end{array}
\right.
\end{equation*}
and
\begin{equation*}
\beta^\eps(t,x)=\left\{
\begin{array}{lll}
\alpha\big(t+2\eps,x\big) \quad& \text{if }\hspace{-0.2cm}&(t,x)\in[-2\eps,T-2\eps]\times\Omega ,\\
0 \quad& \text{if }\hspace{-0.2cm}&(t,x)\in[T-2\eps,T+2\eps]\times\Omega .
\end{array}
\right.
\end{equation*}

Then, we extend $(z^\eps, \beta^\eps)$ by reflection in $Q_T^\eps$, defining, for $(t,x)\in [-2\eps,T+2\eps]\times\Omega_{\bo j}$,
\[
\big(\bar v^\eps(t,x),\bar\alpha^\eps(t,x))=\Big(z^\eps\big(t,\varPhi_{\bo j}(x)\big),\beta^\eps\big(t,\varPhi_{\bo j}(x)\big)\Big) .
\]
We note that the space reflection preserves the weak differentiability.
Hence, $(\bar v^\eps,\bar\alpha^\eps)\in BV(Q_T^\eps)\times L^{q'}(Q_T^\eps)$, $\nabla\bar v^\eps\in L^r(Q_T^\eps)$, and $(\bar v^\eps,\bar\alpha^\eps,\nabla\bar v^\eps)\to(v,\alpha,\nabla v)$ in $L^1(Q_T)\times L^{q'}(Q_T)\times L^r(Q_T)$ as $\eps\to 0$. 

{Similarly, we consider an extension of the Hamiltonian $H$ in the space variable. The extension considered is $\tilde H:U\times\R^N\to\R$, defined as
\begin{equation}\label{extH}
\tilde H(x,p):=H\big(\Phi_{\bo j}(x),p\big)\,.
\end{equation}}

Now, we claim that
\begin{equation}\label{eq:vbar}
-\bar v^\eps_t+{\tilde H(x,\nabla\bar v^\eps)}\le\bar\alpha^\eps\qquad\hbox{ in $Q_T^\eps$ and in the sense of \eqref{eq:distribution}.}
\end{equation}
To prove the claim, we analyze the integrals appearing in \eqref{eq:distribution}. First, recall that $H$ is even in $p_i$ for $1\le i\le n$. Hence, for $x\in\Omega_{\bo j}$,
we have
$$
{\tilde H(x,\nabla\bar v^\eps)=H\Big(\Phi_{\bo j}(x),\nabla z^\eps\big(t,\varPhi_{\bo j}(x)\big)\mathrm{Jac} \varPhi_{\bo j}(x)\Big)=H\Big(\Phi_{\bo j}(x),\nabla z^\eps\big(t,\varPhi_{\bo j}(x)\big)\Big).}
$$
Let $\phi\in\mathcal C_c^\infty((-2\eps,T+2\eps]\times \overline U)$, $\phi\ge0$. Then, {using \eqref{extH},} we have
\begin{gather}\label{est_vbar}
\int_{-2\eps}^{T+2\eps}\int_U\Big(\bar v^\eps\phi_t+{\tilde H(x,\nabla \bar v^\eps)}\phi\Big) dxdt\\
=\sum\limits_{\bo j\in\{1,2,3\}^N}\!\int_{-2\eps}^{T+2\eps}\!\!\!\int_{\Omega_{\bo j}}\!\Big[z^\eps\big(t,\varPhi_{\bo j}(x)\big)\phi_t(t,x)+{H\Big(\Phi_{\bo j}(x),\nabla z^\eps\big(t,\varPhi_{\bo j}(x)\big)\Big)}\phi(t,x)\Big]dxdt .\notag
\end{gather}
Next, we want to use \eqref{eq:distribution} to simplify the right-hand side. However, to achieve this, we have to work with integrals on $\Omega$. Accordingly, we use the change of variable $\varPhi_{\bo j}(x)=y$. The domain $\Omega_{\bo j}$ gets mapped into $\Omega$, because
from \eqref{propvfi}, we have  $x=\varPhi_{\bo j}(y)$, $|\det\mathrm{Jac}\  \varPhi_{\bo j}(x)|=1$. Then, the integral
corresponding to the index $\bo j$ 
in the previous sum becomes:
\begin{align}\label{est_zeta}
&\int_{-2\eps}^{T+2\eps}\into\Big[z^\eps(t,x) \phi_t\big(t,\varPhi_{\bo j}(x)\big)+{H\big(x,\nabla z^\eps(t,x)\big)} \phi\big(t,\varPhi_{\bo j}(x)\big)\Big] dxdt\notag\\
=&\int_{-2\eps}^{T-2\eps} \into \Big[v(t+2\eps,x) \phi_t\big(t,\varPhi_{\bo j}(x)\big) + {H\big(x,\nabla v(t+2\eps,x)\big)}  \phi\big(t,\varPhi_{\bo j}(x)\big)\Big] dxdt\notag\\
+&\int_{T-2\eps}^{T+2\eps}\into\Big[\big(\psi(x)+\lambda(T-2\eps-t)\big) \phi_t\big(t,\varPhi_{\bo j}(x)\big)+{H\big(x,\nabla\psi(x)\big)} \phi\big(t,\varPhi_{\bo j}(x)\big)\Big] dxdt\\
=&\int_{0}^{T}\into\Big[v(s,x) \phi_t\big(s-2\eps,\varPhi_{\bo j}(x)\big)+{H\big(x,\nabla v(s,x)\big)} \phi\big(s-2\eps,\varPhi_{\bo j}(x)\big)\Big] dxds\notag\\
+&\int_{T-2\eps}^{T+2\eps}\into\Big[\big(\psi(x)+\lambda(T-2\eps-t)\big) \phi_t\big(t,\varPhi_{\bo j}(x)\big)+{H\big(x,\nabla\psi(x)\big)}\phi\big(t,\varPhi_{\bo j}(x)\big)\Big] dxdt .\notag
\end{align}
We analyze the two integrals above. For the first one, we use \eqref{eq:distribution} to get
\begin{align*}
&\int_{0}^{T}\into\Big[v(s,x) \phi_t\big(s-2\eps,\varPhi_{\bo j}(x)\big)+{H\big(x,\nabla v(s,x)\big)} \phi\big(s-2\eps,\varPhi_{\bo j}(x)\big)\Big] dxds\\
\le &\intif \alpha(s,x)\phi(s-2\eps,\varPhi_{\bo j}(x)) dxds+\into \psi(x) \phi(T-2\eps,\varPhi_{\bo j}(x)) dx ,
\end{align*}
whereas, for the second one, we integrate by parts in time to get
\begin{align*}
&\int_{T-2\eps}^{T+2\eps}\into\Big[\big(\psi(x)+\lambda(T-2\eps-t)\big) \phi_t\big(t,\varPhi_{\bo j}(x)\big)+{H\big(x,\nabla\psi(x)\big)} \phi\big(t,\varPhi_{\bo j}(x)\big)\Big] dxdt\\
=&\into z^\eps(T+2\eps,x)\phi(T+2\eps,\varPhi_{\bo j}(x)) dx-\into\psi(x)\phi(T-2\eps,\varPhi_{\bo j}(x)) dx\\
+&\int_{T-2\eps}^{T+2\eps}\into\big[\lambda+{H(x,\nabla\psi(x))}\big]\phi(t,\varPhi_{\bo j}(x)) dxdt .
\end{align*}
We choose $\lambda=-\norminf{{H(x,\nabla\psi)}}$ so that the last integral on the right-hand side is non-positive. Plugging these estimates into \eqref{est_zeta}, we have
\begin{align*}
&\int_{-2\eps}^{T+2\eps}\into\Big[z^\eps(t,x) \phi_t\big(t,\varPhi_{\bo j}(x)\big)+{H\big(x,\nabla z^\eps(t,x)\big)} \phi\big(t,\varPhi_{\bo j}(x)\big)\Big] dxdt\\
\le&\intif \alpha(s,x)\phi(s-2\eps,\varPhi_{\bo j}(x)) dxds+\into z^\eps(T+2\eps,x)\phi(T+2\eps,\varPhi_{\bo j}(x)) dx\\
=&\int_{-2\eps}^{T+2\eps}\int_{\Omega_{\bo j}}\beta^\eps(t,\varPhi_{\bo j}(x))\phi(t,x) dxdt+\int_{\Omega_{\bo j}} z^\eps(T+2\eps,x)\phi(T+2\eps,\varPhi_{\bo j}(x)) dx.
\end{align*}
We combine these estimates in  \eqref{est_vbar} by summing over ${\bo j}$. Accordingly, 
we obtain
$$
\iint_{Q_T^\eps}\Big(\bar v^\eps\phi_t+{\tilde H(x,\nabla \bar v^\eps)}\phi\Big) dxdt\le \iint_{Q_T^\eps}\bar\alpha^\eps \phi dxdt+\int_U\bar v^\eps(T+2\eps,x)\phi(T+2\eps,x) dx ,
$$
which proves that $(\bar v^\eps,\bar\alpha^\eps)$ satisfies  \eqref{eq:distribution}  in $Q_T^\eps$.

\emph{Step 3. Regularization and conclusion.} Now, we consider an approximation by convolution. Let $\eta\in\mathcal C_c^\infty(\R^{N+1})$, with $\eta\ge0$, $\int_{\R^{N+1}}\eta(t,x) dxdt=1$ and $\supp(\eta)\subseteq B_1$. We consider, for $\eps>0$, a smooth mollifier $\eta^\eps$ defined as
\begin{equation}\label{mollifier}
\eta^\eps=\eps^{-(N+1)}\eta\left(\frac{(t,x)}\eps\right)
\end{equation} 

We define $(\tilde v^\eps,\tilde \alpha^\eps)=(\bar v^\eps*\eta^\eps,{\bar \alpha^\eps*\eta^\eps+\omega(\eps)})$, {where $\omega(\eps)$ will be defined later}, and prove that $(\tilde v^\eps,\tilde \alpha^\eps)$ satisfies, in the classical sense,
\begin{equation}\label{eq:tildeveps}
-\tilde v^\eps_t+{\tilde H(x,\nabla\tilde v^\eps)}\le\tilde\alpha^\eps ,\qquad\text{in  }\Big\{(t,x)\in Q_T^\eps\ |\ \mathrm{dist}\big((t,x),\partial_PQ_T^\eps\big)>\eps\Big\} .
\end{equation}
We fix $(t,x)\in Q_T^\eps$, with $\mathrm{dist}\big((t,x),\partial_PQ_T^\eps\big)>\eps$, and set $\xi(s,y)=\eta^\eps(t-s,x-y)$. Observe that $\supp(\xi)\subseteq B_\eps((t,x))\subset Q_T^\eps$ and $\supp(\eta^\eps)\subseteq B_\eps$. Using the convexity of $H$ and \eqref{eq:vbar}, we have
\begin{equation}\label{newestimate}
\begin{split}
&-\tilde v^\eps_t(t,x)+{\tilde H(x,\nabla\tilde v^\eps(t,x))}\\
=&-\iint_{Q_T^\eps}\bar v^\eps(s,y)\eta^\eps_t(t-s,x-y) dsdy+{\tilde H\left(x,\iint_{B_\eps}\nabla\bar v^\eps(t-s,x-y)\eta^\eps(s,y) dsdy\right)}\\
\le&\iint_{Q_T^\eps}\bar v^\eps(s,y)\xi_s(s,y) dsdy+\iint_{B_\eps} {\tilde H(x,\nabla\bar v^\eps(t-s,x-y))}\eta^\eps(s,y) dsdy\\
=&\iint_{Q_T^\eps}\Big(\bar v^\eps(s,y)\xi_s(s,y)+{\tilde H(y,\nabla\bar v^\eps(s,y)) \xi(s,y)\Big)}dsdy\\
&{+\iint_{Q_T^\eps}\Big(\tilde H(x,\nabla\bar v^\eps(s,y))-\tilde H(y,\nabla\bar v^\eps(s,y))\Big)\xi(s,y)\,dsdy}\\
\le&\iint_{Q_T^\eps}\bar\alpha^\eps(s,y)\xi(s,y) dsdy {\,+\iint_{Q_T^\eps}\Big(\tilde H(x,\nabla\bar v^\eps(s,y))-\tilde H(y,\nabla\bar v^\eps(s,y))\Big)\xi(s,y)\,dsdy}\,.
\end{split}
\end{equation}
We estimate the last integral. Recalling that $|x-y|<\eps$ if $y\in\supp\xi$, and using Item~\ref{item6.5} of Assumption~\ref{hp2}, we get
\begin{align*}
\left|\iint_{Q_T^\eps}\Big(\tilde H(x,\nabla\bar v^\eps(s,y))-\tilde H(y,\nabla\bar v^\eps(s,y))\Big)\xi(s,y)\,dsdy\right|\le\iint_{Q_T^\eps}|x-y|^\varsigma(1+|\nabla\bar v^\eps|)^\theta\xi\,dsdy\,.
\end{align*}
Using H\"older's inequality with exponents $\frac r\theta$ and $\frac{r}{r-\theta}$ (recall that $r>\theta$), we can bound from above the last integral:
\begin{align*}
\iint_{Q_T^\eps}|x-y|^\varsigma(1+|\nabla\bar v^\eps|)^\theta\xi\,dsdy&\le\eps^{\varsigma}\left(\iint_{Q_T^\eps}(1+|\nabla\bar v^\eps|)^r\,dsdy\right)^{\frac\theta r}\left(\iint_{B_\eps(t,x)}\xi^{\frac r{r-\theta}}\,dsdy\right)^{\frac{r-\theta}r}\\
&\le C\eps^{\varsigma-\frac{\theta(N+1)}r}(1+\norm{\nabla\bar v^\eps})_{L^r(Q_T^\eps)}^{\theta}\,.
\end{align*}
We call
$$
\omega(\eps):=C\eps^{\varsigma-\frac{\theta(N+1)}r}(1+\norm{\nabla\bar v^\eps})_{L^r(Q_T^\eps)}^{\theta}\,.
$$
By Item~\ref{item6.5} of Assumption~\ref{hp2}, we have $\varsigma r>\theta(N+1)$. Thus,  $\omega(\eps)\to0$. Coming back to \eqref{newestimate}, and recalling that $\iint_{Q_T^\eps}\xi=1$, we get
\begin{align*}
-\tilde v^\eps_t(t,x)+\tilde H(x,\nabla\tilde v^\eps(t,x))&\le \iint_{Q_T^\eps}\bar\alpha^\eps(s,y)\xi(s,y) dsdy+\omega(\eps)= \tilde\alpha^\eps(t,x)\,,
\end{align*}
which proves \eqref{eq:tildeveps}. Since $\mathrm{dist}(Q_T,\partial_PQ_T^\eps))=\min\{2\eps,b_1-a_1,\dots,b_n-a_n\}>\eps$, the inequality is also satisfied in $Q_T$ for $\eps$ small enough. Moreover, since the convolution preserves the Sobolev and Lebesgue estimates in the interior of the domain, we have $(\tilde v^\eps,\tilde\alpha^\eps,\nabla\tilde v^\eps)\to(v,\alpha,\nabla v)$ in $L^1([0,T];W^{1,1}(\Omega))\times L^{q'}(Q_T)\times L^r(Q_T)$.

To conclude, we adjust $\tilde v^\eps$ and $\tilde\alpha^\eps$ to satisfy the terminal condition of \eqref{eq:fpgen}. Let $\zeta_\eps\in\mathcal C^\infty(\R)$ with $\zeta_\eps(t)=0$ for $t\le T-\eps$, $\zeta_\eps(t)=1$ for $t\ge T$, and $0\le\zeta_\eps'(t)\le C_0\eps^{-1}$ for some $C_0>0$. We define the terminal approximating functions as
\begin{equation*}
\begin{split}
v^\eps(t,x)=(1-\zeta_\eps(t))\tilde v^\eps(t,x)+\zeta_\eps(t)(\psi(x)+\lambda(T-t)) ,
\end{split}
\end{equation*}
and, for a possibly different constant $C_1>0$,
\begin{equation*}
\alpha^\eps(t,x)=\left\{
\begin{array}{ll}
\tilde\alpha^\eps(t,x) &\text{if }0\le t\le T-\eps ,\\
\tilde\alpha^\eps(t,x)+C_1 &\text{if }T-\eps<t\le T .
\end{array}
\right.
\end{equation*}
Then we have $v^\eps(T,x)=\psi(x)$ and $(v^\eps,\alpha^\eps,\nabla v^\eps)\to (v,\alpha,\nabla v)$ in $L^1([0,T];W^{1,1}(Q_T))\times L^{q'}(Q_T)\times L^r(Q_T)$. We just have to prove that ${-v^\eps_t+\tilde H(x,\nabla v^\eps)}\le\alpha^\eps$. Since $(v^\eps,\alpha^\eps)=(\tilde v^\eps,\tilde\alpha^\eps)$ for $t\le T-\eps$, we consider the case $T-\eps<t\le T$. 

Observe that, for $T-\eps\le t\le T$,
$$
\tilde v^\eps(t,x)=\iint_{Q_T^\eps}\big(\psi(y)+\lambda(T-2\eps-s)\big)\eta^\eps(t-s,x-y) dsdy.
$$
Hence, because of the regularity of $\psi$, $\tilde v^\eps$  is a $\mathcal C^1$ function uniformly in $\eps$. Therefore, we have
\begin{align}
|\tilde v^\eps(t,x)-\psi(x)|&=\left|\iint_{Q_T^\eps}(\bar v^\eps(s,y)-\psi(x))\eta^\eps(t-s,x-y) dsdy\right|\label{eq:stima}\\
&\le\left|\iint_{Q_T^\eps}(\psi(y)-\psi(x))\eta^\eps(t-s,x-y) dsdy\right|+4|\lambda|\eps\le\left(\norminf{\nabla\psi}+4|\lambda|\right)\eps .\notag
\end{align}
Using these estimates, we get
\begin{align*}
&-v^\eps_t+{\tilde H(x,\nabla v^\eps)}\\
=&-(1-\zeta_\eps)\tilde v^\eps_t+{\tilde H\big(x,(1-\zeta_\eps)\nabla\tilde v^\eps+\zeta_\eps\nabla\psi\big)}+\zeta'_\eps(\tilde v^\eps-\psi-\lambda(T-t))+\lambda\zeta_\eps\\
\le&\,(1-\zeta_\eps)\tilde\alpha^\eps-(1-\zeta^\eps){\tilde H(x,\nabla\tilde v^\eps)}+{\tilde H\big(x,(1-\zeta_\eps)\nabla\tilde v^\eps+\zeta_\eps\nabla\psi\big)}+\zeta'_\eps(\tilde v^\eps-\psi-\lambda(T-t))\\
\le&\,\tilde\alpha^\eps+\zeta_\eps {\tilde H(x,\nabla\psi)}+|\zeta'_\eps||\tilde v^\eps-\psi| ,
\end{align*}
where in the last inequality, we used the convexity of $H$ {in the last variable}. If $t\le T-\eps$ the right-hand side is exactly $\tilde\alpha^\eps=\alpha^\eps$. If $T>T-\eps$ we use \eqref{eq:stima} and the bound of $\zeta'_\eps$ to get
\begin{align*}
\tilde\alpha^\eps+\zeta_\eps {\tilde H(x,\nabla\psi)}+\zeta'_\eps|\tilde v^\eps-\psi|
\le\tilde\alpha^\eps+\norminf{{\tilde H(x,\nabla\psi)}}+C_0(\norminf{\nabla\psi}+4|\lambda|) .
\end{align*}
Defining $C_1:=\norminf{{\tilde H(x,\nabla\psi)}}+C_0(\norminf{\nabla\psi}+4|\lambda|)$, we get $-v^\eps_t+{\tilde H(x,\nabla v^\eps)}\le\alpha^\eps$. {Hence, for $x\in\Omega$, $-v^\eps_t+H(x,\nabla v^\eps)\le\alpha^\eps$}. This concludes the proof.
\end{proof}

The second lemma 
needed for the  proof of Theorem \ref{thm:duality}
provides a partial H\"older estimate for functions in $\mathscr K$.

\begin{lem}\label{lem:holder}
Consider the setting of Problem \ref{var3}.
Suppose Assumptions \ref{hp} and \ref{hp2} hold, and let $a,b,\delta\in\R$, $a<b$, $\delta>0$, $V_1\subset\R^N$, and take $(v,\alpha)\in\mathscr K([a,b]\times V_1)$. Then, for $a\le s<t\le b$ and for any $V_0\subset V_1$ with $\mathrm{dist}(V_0,\partial V_1)\ge\delta$, we have
\begin{equation}\label{eq:holder}
v(s,x)\le v(t,x)+C(t-s)^\nu(\|\alpha\|_{L^{q'}}+1)\qquad\text{for }x\in V_0 ,
\end{equation}
where $v(s,\cdot)$ and $v(t,\cdot)$ denote the trace of $v$ in $(s,t)\times V_0$. Here, $C$ depends on $\delta$ and on the constants in Assumption \ref{hp}, and on
$$
\nu:=\frac{r-N(q-1)}{N(q-1)(r-1)+rq}>0 .
$$
\end{lem}

\begin{proof}
We sketch the proof,  slightly modifying \cite[Lemma 3.3 ]{Card1order}. Assume $v,\alpha\in\mathcal C^1$. This hypothesis is removed at the end.

Because of Item \ref{item7} of Assumption \ref{hp2}, we can take $\beta\in(r^{-1},(N(q-1))^{-1})$. We also choose $x\in V_0$ and $\sigma\in\R^N$, with $|\sigma|<\delta$, and we define
\begin{equation*}
x_\sigma(\tau)=\left\{\begin{array}{lr}
x+\sigma(\tau-s)^\beta &\text{if }s\le\tau\le\frac{s+t}2 ,\\
x+\sigma(t-\tau)^\beta &\text{if }\frac{s+t}2\le\tau\le t .
\end{array}
\right.
\end{equation*}
Given the assumptions in the Lemma statement,  $x_\sigma(\tau)\in V_1$ for all $\tau\in[s,t]$. Moreover, using \eqref{eq:fpgen} and the definition of Legendre transform \eqref{legtransH},
we have
$$
\frac d{d\tau}\left[v(\tau,x_\sigma(\tau))-\int_{\tau}^t {L(x_\sigma(l),x_\sigma'(l))} dl\right]\ge-\alpha(\tau,x_\sigma(\tau)) .
$$
Computing the integral for $(\tau,\sigma)\in[s,t]\times B_\delta$, we find
$$
v(s,x)\le v(t,x)+\frac 1{|B_\delta|}\int_s^t\int_{B_\delta}\big[{L(x_\sigma(\tau),x'_\sigma(\tau))}+\alpha(\tau,x_\sigma(\tau))\big] d\sigma d\tau .
$$
Considering the assumptions on $L$ and integrating, we estimate the Lagrangian term
$$
\int_s^t\int_{B_\delta}{L(x_\sigma(\tau),x'_\sigma(\tau))}\, d\sigma d\tau\le C_{\delta}(t-s)^{1+r'(\beta-1)} .
$$
Thus, using H\"older's inequality and a suitable change of variable, we can estimate the last term on the right-hand side as
$$
\int_s^t\int_{B_1}\alpha(\tau,x_\sigma(\tau)) d\sigma d\tau\le C_\delta(t-s)^{\frac{1-n\beta(q-1)}q}\|\alpha\|_{L^{q'}}.
$$
Finally, by choosing $\beta=\frac{q+r-1}{n(q-1)(r-1)+qr}$, we obtain \eqref{eq:holder} in the regular case. We obtain the result in the general case by approximation with $\mathcal C^\infty$ functions, as in the previous Lemma.
\end{proof}

We proceed with the proof of Theorem \ref{thm:duality}.

\begin{proof}[Proof of Theorem \ref{thm:duality}]
If $v\in\mathscr I$, then 
$$
(v,-v_t+{H(x,\nabla v)})\in\mathscr K ,\qquad \text{and}\quad\mathcal I(v)=\mathcal K(v,-v_t+{H(x,\nabla v)}) .
$$
This implies the inequality
$$
\inf\limits_{(v,\alpha)\in\mathscr K}\mathcal K(v,\alpha)\le\inf\limits_{v\in\mathscr I}\mathcal I(v) .
$$
To prove the opposite inequality,
we take $(v,\alpha)\in\mathscr K$. 
Since $F^*(t,x,a)=0$ for $a\le 0$, we have $F^*(t,x,\alpha)=F^*(t,x,\alpha^+)$, where $\alpha^+=\max\{\alpha,0\}$ denotes the positive part of $\alpha$. Then $\mathcal K(v,\alpha)= \mathcal K\big(v,\alpha^+\big)$ and we can replace $\alpha$ with $\alpha^+$. Thus, we can assume $\alpha(t,x)\ge 0$ for all $(t,x)\in Q_T$.
For $\eps>0$, consider $(v^\eps,\alpha^\eps)\in\mathscr I\times\mathcal C^\infty(\Omega)$ defined as in Lemma \ref{lem:regularization}. Then, $(v^\eps,\alpha^\eps)\to (v,\alpha)$ in $L^1([0,T];W^{1,1}(\Omega))\times L^{q'}(\Omega)$, $\alpha^\eps\ge0$ and we have
\begin{equation}\label{eq:reverse}
\inf\limits_{z\in\mathscr I} \mathcal I(z)\le \intif F^*(t,x,-v^\eps_t+{H(x,\nabla v^\eps))} dxdt-\into v^\eps(0,x) dm_0(x)-\int_0^T\int_{\partial\Omega}jv^\eps dx dt .
\end{equation}
Since $F^*$ is non-decreasing in the last variable, we have
\begin{equation}\label{eq:primaF}
\begin{split}
\limsup\limits_{\eps\to0}\intif F^*(t,x,-v^\eps_t+{H(x,\nabla v^\eps))} dxdt&\le\limsup\limits_{\eps\to0}\intif F^*(t,x,\alpha^\eps) dxdt\\&\le \int_0^T\int_\Omega F^*(t,x,\alpha) dxdt ,
\end{split}
\end{equation}
where the last inequality follows from the upper bound on $F^*$ given by \eqref{hp:FFstar} and the Fatou's Lemma applied to the sequence $\{-F^*(t,x,\alpha^\eps)\}_\eps$.

For the second integral, we apply Lemma \ref{lem:holder}. Define $Q_T^\eps$ and $\varPhi_{\bo j}$ as in \eqref{QTeps} and \eqref{Phij}, and the mollifier $\eta^\eps$ as in \eqref{mollifier}. From the definition of $v^\eps$, we have, for a constant $C$ depending on $\|\alpha\|_{L^{p'}}$,
\begin{align*}
v^\eps(0,x)&=\iint_{Q_T^\eps} v\big(t+2\eps,\varPhi_{\bo j}(y)\big)\eta^\eps(-s,x-y) dyds\\
&\ge\iint_{Q_T^\eps}v(0,\varPhi_{\bo j}(y))\eta^\eps(-s,x-y) dyds-C\eps\to v(0,x) ,
\end{align*}
where the convergence holds in $L^1$ since $v\in BV$ ensures $v(0,\cdot)\in L^1(\Omega)$. Hence,
\begin{equation}\label{liminf}
\limsup\limits_{\eps\to0}\left(-\into v^\eps(0,x) m_0(x) dx\right)\le -\into v(0,x) m_0(x) dx .
\end{equation}
For the last integral, we note that
$$
v^\eps-v\in L^1([0,T];W^{1,1}(\Omega))\implies v^\eps(t,\cdot)-v(t,\cdot)\in W^{1,1}(\Omega)\text{ for a.e. }t\in[0,T] .
$$
Hence, for a.e. $t$ the trace of $v^\eps(t,\cdot)-v(t,\cdot)$ is well defined at $\partial\Omega$, and we have
$$
\int_{\partial\Omega}|v^\eps(t,x)-v(t,x)| dx\le C\|v^\eps(t,\cdot)-v(t,\cdot)\|_{W^{1,1}(\Omega)} .
$$
Integrating in time, we find
$$
\intb|v^\eps-v| dxdt\le C\|v^\eps-v\|_{L^1(W^{1,1})}\to0 ,
$$
which ensures
\begin{equation}\label{eq:convJ}
\intb jv^\eps dxdt\to\intb jv dxdt .
\end{equation}
Using \eqref{eq:primaF}, \eqref{liminf} and \eqref{eq:convJ} in \eqref{eq:reverse}, we find
$$
\inf\limits_{z\in\mathscr I}\mathcal I(z)\le \intif F^*(t,x,\alpha) dxdt-\into v(0,x) dm_0(x)-\int_0^T\int_{\partial\Omega}jv dx dt=\mathcal K(v,\alpha) .
$$
Taking the infimum over $(v,\alpha)\in\mathscr K$, $\alpha\ge0$ we find
$$
\inf\limits_{v\in\mathscr I}\mathcal I(v)\le\inf\limits_{(v,\alpha)\in\mathscr K}\mathcal K(v,\alpha) .
$$
This implies \eqref{FinalDuality} and, together with \eqref{eq:duality}, concludes the Theorem.
\end{proof}

We conclude this section with the following corollary, which plays a role in Section \ref{sec7} to prove the existence of solutions for Problem \ref{problem_mfg}.

\begin{cor}\label{cor:dual}
Consider the setting of Problems \ref{var2} and \ref{var3}.
Let $(m,w)\in\mathscr M$ and $(v,\alpha)\in\mathscr K$, and suppose Assumptions \ref{hp} and \ref{hp2} hold. Then we have
\begin{equation}\label{dual_leq}
\begin{split}
\into v(0) m_0 dx+\intb jv dxdt\le\into m(T)\psi dx+\intif \left[m{L\left(x,\frac wm\right)}+ m\alpha\right]dxdt .
\end{split}
\end{equation}
\end{cor}
\begin{proof}
Suppose the right-hand side is bounded, otherwise the inequality is obvious. Arguing as in Proposition \ref{prop_regM}, we have $w\in L^\beta$, with $\beta>1$.

Thanks to Lemma \ref{lem:regularization}, we take $(v^\eps,\alpha^\eps)\in\mathscr I\times\mathcal C^\infty$ such that $(v^\eps,\alpha^\eps)\to(v,\alpha)$ in $L^1([0,T];W^{1,1}(Q_T))\times L^{q'}(Q_T)$. Using $v^\eps$ as a test function for the pair $(m,w)$, we can apply \eqref{generalformula} since $(m,w)\in L^\theta(Q_T;\R^{N+1})$, with $\theta=\frac{qr}{qr-q+1}>1$. We get, for all $\eps>0$,
\begin{equation*}
\begin{split}
\into m(0)v^\eps(0) dx+\intif \big(mv^\eps_t+w\cdot\nabla v^\eps\big)dxdt+\intb jv^\eps dxdt=\into m(T)\psi dx .
\end{split}
\end{equation*}
Observe that the integral $\into m(T)\psi dx$ makes sense thanks to Proposition \ref{pro33}. Since $(v^\eps,\alpha^\eps)$ is a classical solution of \eqref{eq:fpgen}, substituting we find
\begin{equation*}
\begin{split}
\into m(0)v^\eps(0) dx+\intif \big(m{H(x,\nabla v^\eps)}+w\cdot\nabla v^\eps\big)dxdt&+\intb jv^\eps dxdt\\
&\le\into m(T)\psi dx+\intif m\alpha^\eps dxdt .
\end{split}
\end{equation*}
From the definition of $L$ and from \eqref{liminf}, we have
\begin{equation}\label{eqL}
\begin{split}
\liminf_{\eps\to0}\Bigg[\intif \big(w\cdot\nabla v^\eps&+m{H(x,\nabla v^\eps)}\big) dxdt+\into v^\eps(0) m_0(x) dx\Bigg]\\
\ge &-\intif m{L\left(x,\frac wm\right)}dxdt+\into v(0) m_0(x) dx .
\end{split}
\end{equation}
Hence
\begin{equation*}
\begin{split}
\into m(0)v(0) dx+\intif jv^\eps dxdt\le \into m(T)\psi dx+\intif\left[m{L\left(x,\frac wm\right)}+m\alpha^\eps\right] dxdt .
\end{split}
\end{equation*}
Since $v^\eps\to v$ in $L^1(W^{1,1})$ and $\alpha^\eps\to\alpha$ in $L^{q'}$, $m\in L^q$, we can pass to the limit in the remaining two integrals and obtain \eqref{dual_leq}.
\end{proof}

%% file: Existence.tex
The following proposition establishes the existence of minimizers for Problem \ref{var3}. Since Problem \ref{var3} is a relaxation of Problem \ref{var1}, we obtain a relaxed solution for the latter as a consequence. We prove this result under Assumptions \ref{hp}, \ref{hp2}, and \ref{hp3}.
Assumption \ref{hp2} is only required in the first step of Theorem \ref{thm:duality} for technical reasons. 
Therefore, if Theorem \ref{thm:duality} can be proved without relying on Assumption \ref{hp2}, then neither the specific choice of $\Omega$ nor the evenness of the Hamiltonian $H$ is required. This is because all other results only assume that $\Omega$ is a general Lipschitz domain.

\begin{pro}\label{pro:ex_relax}
	Consider the setting of Problem \ref{var3}.
	Suppose Assumptions \ref{hp}, \ref{hp2}, and \ref{hp3} hold. Then there exists $(\bar v,\bar\alpha)\in\mathscr K$ such that
	$$
	\inf\limits_{(v,\alpha)\in\mathscr K}\mathcal K(v,\alpha)=\min\limits_{(v,\alpha)\in\mathscr K}\mathcal K(v,\alpha)=\mathcal K(\bar v,\bar\alpha)\,.
	$$
\end{pro}
\begin{proof}
	By Theorem \ref{thm:duality}, 
	we can select a minimizing sequence $v^n\in\mathscr I$ such that \[\mathcal I(v^n)\to\inf\limits_{(v,\alpha)\in\mathscr K}\mathcal K(v,\alpha).\] Let  $\alpha^n:=\max\{-v_t^n+{H(x,\nabla v^n)},\,f(t,x,0)\}$. By Remark \ref{rem23}, $F^*(t,x,a)=0$ for $a\le f(t,x,0)$. Thus,  we have
	$$
	F^*(t,x,\alpha^n)=F^*(t,x,-v^n_t+{H(x,\nabla v^n)}).
	$$
	Accordingly, 
	$$
    \mathcal I(v^n)=\mathcal K(v^n,\alpha^n). 
	$$
	Consequently, $(v^n,\alpha^n)\in\mathscr K$ is a minimizing sequence of $\mathcal K$. 
	To pass to the limit and obtain the existence of a minimizer, we need uniform bounds on $v^n$ and $\alpha^n$. Accordingly, next, we prove a $W^{1,1}(Q_T)$ bound for $v^n$ and a $L^{q'}(Q_T)$ bound for $\alpha^n$.  
	
	We start by observing that $\mathcal K(v^n,\alpha^n)\le C$ for some $C>0$ independent of $n$, because $(v^n,\alpha^n)$ is a minimizing sequence. This means
	\begin{equation}\label{upper_bound}
		\intif F^*(t,x,\alpha^n)\,dxdt-\into v^n(0)\,dm_0(x)-\intb j(t,x)\,v^n(t,x)\,dxdt\le C\,.
	\end{equation}
	In particular, because $F^*\ge0$, we have
	\begin{equation}\label{camundongo}
		-\into v^n(0)\,dm_0(x)-\intb j(t,x)\,v^n(t,x)\,dxdt\le C\,.
	\end{equation}
	Now, we use \eqref{camundongo} to prove a $L^1(\partial_PQ_T)$ bound for $v^n$ and \eqref{upper_bound} to prove a $L^{q'}$ bound for $\alpha^n$, uniformly in $n\in\N$. We consider the sets
	$$
	A_n=\big\{x\in\Omega\,|\ \,v^n(0,x)\ge0\big\}\,,\qquad B_n=\big\{(t,x)\in [0,T]\times\partial\Omega\,|\,\ v^n(t,x)\ge0\big\}
	$$
	and $A_n^c$, $B_n^c$ as the complements of $A_n$ and $B_n$ with respect to $\Omega$ and $[0,T]\times\partial\Omega$. We get
	\begin{align*}
		\iint_{\partial_PQ_T}|v^n|\,dxdt=\into|\psi|\,&dx+\into|v^n(0)|\,dx+\intb|v^n|\,dxdt\\
		\le C \,+ \int_{A_n}v^n(0)\,&dx-\int_{A_n^c}v^n(0)\,dx+\iint_{B_n}v^n\,dxdt-\iint_{B_n^c}v^n\,dxdt\,.
	\end{align*}
	Here, $C$ denotes a generic positive constant that may vary from line to line.
	
	We analyze each term. Thanks to \eqref{eq:holder} in Lemma \ref{lem:holder}, we have
	\begin{equation}\label{pos_est}
		\int_{A_n}v^n(0)\,dx+\iint_{B_n}v^n\,dxdt\le \iint_{A_n\cup B_n} (\psi(x)+T\|\alpha^n\|_{L^{q'}})\,dxdt\le C\big(1+\|\alpha^n\|_{L^{q'}}\big)\,.
	\end{equation}
	For the other terms, we recall that, thanks to Item~\ref{item8} of Assumption \ref{hp3}, we have $C^{-1}<m_0<C$ and $C^{-1}<j<C$ for some $C>0$. Then, using \eqref{camundongo}, we get
	\begin{align*}
		&-\int_{A_n^c}v^n(0)\,dx-\iint_{B_n^c}v^n\,dxdt\le C\left(-\int_{A_n^c}v^n(0)m_0(x)\,dx-\iint_{B_n^c}j\,v^n\,dxdt\right)\\
		=\, &C\left(-\into v^n(0)m_0(x)\,dx-\intb j\,v^n\,dxdt+\int_{A_n}v^n(0)m_0(x)\,dx+\iint_{B_n}j\,v^n\,dxdt\right)\\
		\le\,&C\left(1+\int_{A_n}v^n(0)\,dx+\iint_{B_n}v^n\,dxdt\right)\le C\big(1+\|\alpha^n\|_{L^{q'}}\big)\,,
	\end{align*}
	where in the last inequality, we used \eqref{pos_est}. This implies
	\begin{equation}\label{boundary}
		\iint_{\partial_PQ_T}|v^n|\,dxdt\le C\big(1+\|\alpha^n\|_{L^{q'}}\big)\,.
	\end{equation}
	Returning to  \eqref{upper_bound}, the estimate \eqref{boundary} implies
	$$
	\intif F^*(t,x,\alpha^n)\,dxdt\le C\big(1+\|\alpha^n\|_{L^{q'}}\big).
	$$
	Using the growth properties of $F^*$ stated in \eqref{hp:FFstar}, we deduce
	$$
	\|\alpha^n\|_{L^{q'}}\le C\,,
	$$
	which implies, again from \eqref{boundary}, $\|v^n\|_{L^1(\partial_PQ_T)}\le C$.
	
	Now, we obtain a bound for $\nabla v^n$ in $L^r$ as follows. Using Item~\ref{item2} of Assumption~\ref{hp}, we get
	\begin{align*}
		\intif|\nabla v^n|^r\,dxdt&\le C+C\intif {H(x,\nabla v^n)}\,dxdt\le C+C\intif (\alpha^n+v^n_t)\,dxdt\\
		&\le C+C\|\alpha^n\|_{L^{q'}}+\into\psi(x)\,dx-\into v^n(0)\,dx\le C\,,
	\end{align*}
	where in the last inequality, we used \eqref{boundary} and the $L^{q'}$ bound of $\alpha^n$.
	
	Finally, we get a bound for $v^n_t$ as follows. First, we note that $|v^n_t|=v^n_t+2(v^n_t)^-$.
	Thus,  we have
	\begin{align*}
		\intif |v^n_t|\,dxdt=&\into(\psi(x)-v^n(0))+2\iint_{\{v^n_t<0\}}(-v^n_t)\,dxdt\\
		\le C\,+2&\iint_{\{v^n_t<0\}}(-v^n_t+{H(x,\nabla v^n)})\,dxdt\le C\big(1+\|\alpha^n\|_{L^{q'}}\big)\le C\,,
	\end{align*}
	where we used the fact that $H\ge0$, $-v^n_t+{H(x,\nabla v^n)}\le\alpha^n$ and the $L^{q'}$ bound of $\alpha^n$. Because
	$v^n_t$ is bounded in $L^1$ and $v^n(T)=\psi$, we deduce that $v^n$ is bounded in $L^1(Q_T)$. Although $v^n$ is actually in $\mathcal C([0,T];L^1(\Omega))$, this stronger  result is not required here. Hence, we have proved
	\[
	\|v^n\|_{W^{1,1}(Q_T)}+\|\alpha^n\|_{L^{q'}(Q_T)}\le C\,,
	\]
	which implies the existence of a pair $(\bar v,\bar\alpha)\in BV\times L^{q'}$ such that $v^n\to \bar v$ in $L^1(Q_T)$, $v^n_t\weak v_t$ in the sense of measures, $\alpha^n\weak\bar\alpha$ in $L^{q'}$, up to a non-relabelled subsequence. Moreover,
	since $\nabla v^n$ is bounded in $L^r$, it follows that, up to a subsequence, $\nabla v^n$ converges weakly to $\nabla\bar v$ in $L^r$.
	
	To prove that $(\bar v,\bar\alpha)\in\mathscr K$, we have to verify that  \eqref{eq:distribution} holds. Since $v^n$ is a (classical) solution of \eqref{eq:fpgen}, for each $\phi\in\mathcal C_c^\infty((0,T]\times\overline\Omega)$, $\phi\ge0$, we have
	$$
	\intif \Big(v^n\phi_t+{H(x,\nabla v^n)}\phi\Big)\,dxdt\le\intif\alpha^n\phi\,dx dt+\into \psi(x)\phi(T,x)\,dx\,.
	$$
	Since $H$ is convex {in the last variable} and continuous by Item~\ref{item2} of Assumption~\ref{hp}, it is lower semicontinuous for the weak convergence, hence
	$$
	\intif {H(x,\nabla\bar v)}\,\phi\,dxdt\le\liminf\limits_{n\to+\infty}\intif {H(x,\nabla v^n)}\,\phi\,dxdt\,.
	$$
	For the remaining terms, we can pass to the limit thanks to the weak convergence of $(v^n,\alpha^n)$, and we obtain \eqref{eq:distribution} for the pair 
	$(\bar v,\bar\alpha)$. Hence, $(\bar v,\bar\alpha)\in\mathscr K$.
	
	To conclude the theorem, we must prove that $(\bar v,\bar\alpha)$ minimizes $\mathcal K$. We start from
	\begin{equation}\label{eq:infconv}
		\lim\limits_{n\to+\infty}\left(\intif F^*(t,x,\alpha^n)\,dxdt-\into v^n(0)\,dm_0(x)-\intb j\,v^n\,dxdt\right)=\inf\limits_{(v,\alpha)\in\mathscr K}\mathcal K(v,\alpha)\,.
	\end{equation}
	Because $F^*$ is convex and continuous, as before, we have
	$$
	\intif F^*(t,x,\bar\alpha)\,dxdt\le\liminf\limits_{n\to+\infty}\intif F^*(t,x,\alpha^n)\,dxdt\,.
	$$
	For the second term, we use \eqref{eq:holder} in Lemma \ref{lem:holder} to get, for all $s\in(0,T]$,
	$$
	-\into v^n(0)\,dm_0(x)\ge-\into v^n(s)\,dm_0(x)-Cs^\nu\,.
	$$
	Integrating both sides of the preceding inequality in $s\in[0,\eps]$ and dividing by $\eps$, we find
	$$
	-\into v^n(0)\,dm_0(x)\ge-\frac1\eps\int_0^\eps\into v^n(s)\,dm_0(x)-C\eps^\nu\,.
	$$
	Taking $\liminf$ in $n$ and using the weak convergence of $v^n$, we get
	\begin{equation}\label{limsup}
		\liminf\limits_{n\to+\infty}\left(-\into v^n(0)\,m_0\,dx\right)\ge-\frac1\eps\int_0^\eps\into \bar v(s)\,m_0\,dx-C\eps^\nu\overset{\eps\to0}{\longrightarrow}-\into \bar v(0)\,m_0\,dx\,,
	\end{equation}
	where the convergence in $\eps$ holds because $\bar v\in BV(Q_T)$.
	
	Now, we analyze the boundary integral $\intb jv^n\,dxdt$.
	Because $v^n$ is bounded in $W^{1,1}(Q_T)$, we have from Sobolev's embedding theorem that $\|v^n\|_{L^{N'}}\le C$, where $N'=1^*=\frac N{N-1}$ if $N>1$ (for $N=1$ the bound holds in $L^p$ for all $p\ge1$). Because $\nabla v^n$ is bounded in $L^r$, we have
	$$
	\|v^n\|_{L^s([0,T];W^{1,s}(\Omega))}\le C\,,\qquad\text{where }s=N'\wedge r\,.
	$$
	Because $L^s(W^{1,s})$ is a reflexive Banach space, there exists $\bar v\in L^s(W^{1,s})$ such that $v^n\weak\bar v$ in $L^s(W^{1,s})$ up to subsequences.
	
	The functional mapping $L^s([0,T];W^{1,s}(\Omega))$ to $\Rr$ given by
	$$
	z\mapsto  \intb z(t,x)\,j(t,x)\,dxdt
	$$
is an element of $\left(L^s(W^{1,s})\right)'$. Thus, 
	$$
	\intb j(t,x)\,v^n(t,x)\,dxdt\to\intb j(t,x)\,\bar v(t,x)\,dxdt\,.
	$$
	Coming back to \eqref{eq:infconv}, we have proved that
	$$
	\mathcal K(\bar v,\bar\alpha)\le\inf\limits_{(v,\alpha)\in\mathscr K}\mathcal K(v,\alpha)\,,
	$$
	which concludes the proof.
\end{proof}

The following corollary provides additional properties of the minimizer necessary for the proof of Theorem \ref{thm:main}.

\begin{cor}\label{cor:strong_grad}
	Consider the setting of Problem \ref{var3}.
	Suppose Assumptions \ref{hp}, \ref{hp2} and \ref{hp3} hold. If $(v,\alpha)\in\mathscr K$ minimizes $\mathcal K$, there exists a minimizing sequence $(v^n,\alpha^n)\in\mathscr I\times\mathcal C^\infty(Q_T)$ such that 
	$$
	(v^n,\alpha^n,\nabla v^n)\to (v,\alpha,\nabla v),\qquad\text{strongly in }L^1([0,T];W^{1,1}(\Omega))\times L^{q'}(Q_T)\times L^r(Q_T)\,.
	$$
	Moreover, we have $v(T)=\psi$ in the sense of traces and $v^n(0)\weak v(0)$ weakly in $L^1(\Omega)$.
\end{cor}
\begin{proof}
	If there exists a minimum $(v,\alpha)$ of $\mathcal K$, we take as a minimizing sequence the sequence $(v^n,\alpha^n)=(v^{\eps_n},\alpha^{\eps_n})$, where $\eps_n=\frac1n$ and $(v^\eps,\alpha^\eps)$ is defined in Lemma \ref{lem:regularization}.
	According to the same Lemma, we have $(v^n,\alpha^n)\in\mathscr I\times\mathcal C^\infty(Q_T)$.
	In particular, we also know from the Lemma that $(v^n,\alpha^n,\nabla v^n)\to(v,\alpha,\nabla v)$ in $L^1([0,T];W^{1,1}(\Omega))\times L^{q'}(Q_T)\times L^r(Q_T)$.
	
	To prove that $v^n(0)\weak v(0)$, we use \eqref{liminf} and \eqref{limsup}. In those results, we just used that $m_0$ is an $L^\infty$ non-negative function. Hence, for all $\phi\in L^\infty(\Omega)$ such that $\phi\ge0$ a.e., we have
	\begin{align*}
		&\limsup\limits_{n\to+\infty}\left(-\into v^n(0)\,\phi\,dx\right)\le -\into v(0)\,\phi\,dx\implies \liminf\limits_{n\to+\infty}\into v^n(0)\,\phi\,dx\ge \into v(0)\,\phi\,dx\\
		&\liminf\limits_{n\to+\infty}\left(-\into v^n(0)\,\phi\,dx\right)\ge-\into v(0)\,\phi\,dx\implies \limsup\limits_{n\to+\infty}\into v^n(0)\,\phi\,dx\le \into v(0)\,\phi\,dx\,.
	\end{align*}
	These two inequalities  imply $v^n(0)\weak v(0)$ in $L^1(\Omega)$.
	
	To conclude, we prove that $v(T)=\psi$ in the sense of traces. For $\phi\in\mathcal C^\infty(\overline\Omega)$, we have
	$$
	\into v(T)\phi\,dx=\into v(0)\phi\,dx+\intif \phi\, dv_t\,.
	$$
	Because $v^n_t\weak v_t$ in the sense of measures and $v^n(0)\weak v(0)$ in $L^1$, we have
	$$
	\into v(T)\phi\,dx=\lim_{n\to+\infty}\left(\into v^n(0)\phi\,dx+\intif \phi\, dv^n_t\right)=\lim\limits_{n\to+\infty}\into v^n(T)\phi\,dx=\into\psi\,\phi\,dx\,,
	$$
	which implies $v(T)=\psi$ in the sense of traces and concludes the proof.
\end{proof}

%% file: MFG.tex
We now turn our attention to the main results of this paper, beginning with an existence result for Problem \ref{problem_mfg}.
\begin{teo}\label{teo:ex_mfg}
	Suppose Assumptions \ref{hp}, \ref{hp2}, and \ref{hp3} hold, and let $(m,w)\in\mathscr M$ and $(u,\alpha)\in\mathscr K$ be solutions of Problems \ref{var2} and \ref{var3} respectively. 
	Then, almost everywhere on the set where $m>0$, we have $\alpha(t,x)=f(t,x,m(t,x))$ and $w(t,x)=-m(t,x)H_p(t,x,\nabla u(t,x))$.
	Furthermore, $(u,m)$ is a weak solution of Problem \ref{problem_mfg} in the sense of Definition \ref{defmfg}.
\end{teo}

\begin{proof}
	Let $(u,\alpha)$ and $(m,w)$ be as in the statement of the theorem. By Remark \ref{rem23}, $F^*(t,x,a)=0$ for $a\le f(t,x,0)$. 
	Therefore, we may assume without loss of generality that $\alpha(t,x)\ge f(t,x,0)$ for all $(t,x)\in Q_T$ by replacing $\alpha$ with $\max\{\alpha(t,x), f(t,x,0)\}$. From \eqref{FinalDuality}, we have
	$$
	\mathcal K(u,\alpha)+\mathcal M(m,w)=0 .
	$$
	Accordingly,
	\begin{align*}
		\intif \big(F^*(t,x,\alpha)+F(t,x,m)\big) dxdt=&\into u(0) m_0 dx-\into\psi(x) m(T,dx)\\
		&+\intb ju dxdt-\intif m{L\left(x,\frac wm\right)}dxdt .
	\end{align*}
	For the left-hand side, the definition of $F^*$ implies
	$$
	\intif \alpha m dxdt\le\intif \big(F^*(t,x,\alpha)+F(t,x,m)\big) dxdt .
	$$
	For the right-hand side, \eqref{dual_leq} states that
	$$
	\into u(0)m_0 dx-\into m(T)\psi dx+\intb ju dxdt-\intif m {L\left(x,\frac wm\right)}dxdt\le\intif m\alpha dxdt .
	$$
	Comparing both sides, we notice that
	we have a chain of inequalities that starts and ends with the same term. 
	Therefore, all these inequalities must actually be equalities. Consequently, \eqref{dual_eq} holds, and furthermore,
	$$
	\intif \alpha m dxdt=\intif \big(F^*(t,x,\alpha)+F(t,x,m)\big) dxdt.
	$$
	Thus,
	$$
	\alpha m=F^*(t,x,\alpha)+F(t,x,m) ,\quad a.e..
	$$
	From the last equality it follows that $\alpha$ and $m$ are conjugated variables.
	Since $F$ is differentiable in $m$ on the set where $m>0$, it follows that
	$$
	\alpha(t,x)=F_m(t,x,m(t,x))=f(t,x,m(t,x))\qquad\text{a.e. }(t,x)\in \{m>0\} .
	$$
	If $m(t,x)=0$, by \eqref{vanish_Fstar}, we have $F^*(t,x,\alpha(t,x))=0$. This implies $\alpha(t,x)\le f(t,x,0)$. Because, by assumption, $\alpha(t,x)\ge f(t,x,0)$, we get $\alpha(t,x)=f(t,x,0)$.
	
	Consequently, we have
	\begin{equation}\label{cost_true}
		\alpha=f(t,x,m) .
	\end{equation}
	Using \eqref{dual_eq} and the proof of Corollary \ref{cor:dual}, we obtain
	that equality must hold in \eqref{eqL}. Therefore, the liminf appearing in \eqref{eqL} is a limit. Next, we consider the minimizing sequence from Corollary \ref{cor:strong_grad}. Observing from the same result that \( v^n(0) \weak v(0) \) in \( L^1(\Omega) \), we have
		\begin{equation}\label{eq:lim_int}
			\intif m {L\left(x,\frac{w}{m}\right)} dxdt = \lim_{n\to+\infty} \intif m \left( -\frac{w}{m} \cdot \nabla v^n - {H(x,\nabla v^n)} \right) dxdt.
		\end{equation}
		Consequently,
		\begin{equation}\label{eq:lim_zero}
			\lim_{n\to+\infty} \intif m \left[ {L\left(x,\frac{w}{m}\right)} + \frac{w}{m} \cdot \nabla v^n + {H(x,\nabla v^n)} \right] dxdt = 0.
		\end{equation}
		
		Moreover, since \( w \in L^\beta \) with \( \beta = \frac{q r}{q r - q + 1} \) and \( \nabla v^n \) is uniformly bounded in \( L^r \), converging weakly to \( \nabla u \) in \( L^r \), we obtain
		\begin{equation}
		\label{eq:conv_w}
		\lim_{n\to \infty} \intif w \cdot \nabla v^n \, dxdt = \intif w \cdot \nabla u \, dxdt,
		\end{equation}
		provided \( q \leq r' \). This inequality holds under the conditions specified in Assumption \ref{hp4}.
		
		Fix \( M > 0 \). Then,
		\begin{align*}
			\lim_{n\to+\infty} \intif m H(x,\nabla v^n) \, dxdt &\geq \lim_{n\to+\infty} \int_0^T \int_{m + |\nabla u| < M} m H(x,\nabla v^n) \, dxdt \\
			&\geq \int_0^T \int_{m + |\nabla u| < M} m H(x,\nabla u) \, dxdt \\
			&\quad + \lim_{n\to+\infty} \int_0^T \int_{m + |\nabla u| < M} m H_p(x,\nabla u) \cdot (\nabla v^n - \nabla u) \, dxdt \\
			&= \int_0^T \int_{m + |\nabla u| < M} m H(x,\nabla u) \, dxdt,
		\end{align*}
		where the last equality follows because the weak convergence of \( \nabla v^n \) to \( \nabla u \) implies that the final integral tends to zero.
		
		Letting \( M \to \infty \), we deduce that
		\begin{equation}\label{eq:H_limit}
			\lim_{n\to+\infty} \intif m H(x,\nabla v^n) \, dxdt \geq \intif m H(x,\nabla u) \, dxdt.
		\end{equation}
		Therefore, combining equations \eqref{eq:lim_zero}, \eqref{eq:conv_w}, and \eqref{eq:H_limit}, we obtain
		\begin{equation}\label{eq:ineq_zero}
			\intif m \left[ {L\left(x,\frac{w}{m}\right)} + \frac{w}{m} \cdot \nabla u + {H(x,\nabla u)} \right] dxdt \leq 0.
		\end{equation}
		Since \( L \) and \( H \) are convex conjugate functions, we have yields
		\[
		{L\left(x,\frac{w}{m}\right)} + {H(x,\nabla u)} \geq -\frac{w}{m} \cdot \nabla u,
		\]
		which holds pointwise for the integrand. Consequently, the inequality in \eqref{eq:ineq_zero} can only be satisfied if the integrand vanishes almost everywhere, that is,
		\[
		{L\left(x,\frac{w}{m}\right)} + \frac{w}{m} \cdot \nabla u + {H(x,\nabla u)} = 0 \quad \text{a.e. in } Q_T.
		\]
		This implies
		\[
		w = -m {H_p(x,\nabla u)} \quad \text{a.e. in } \{ m > 0 \}.
		\]
		Since \( w = 0 \) almost everywhere in the set \( \{ m = 0 \} \), it follows that
		\begin{equation}\label{drift_true}
			w = -m {H_p(x,\nabla u)} \quad \text{a.e. in } Q_T.
		\end{equation}
		By substituting \eqref{cost_true} and \eqref{drift_true} into the equations satisfied by \( (u,\alpha) \) and \( (m,w) \), we conclude that \( (u,m) \) satisfies Item~\ref{solitem1} and Item~\ref{solitem3} of Definition \ref{defmfg}. Furthermore, since \( {L\left(x,\frac{w}{m}\right)} \in L^1(m) \) and \eqref{drift_true} holds, Item~\ref{solitem4} follows directly from the growth assumptions on \( H \) and \( H^* \).

	To prove Item~\ref{solitem2}, we denote by $u_t^{ac}$ and $u_t^s$ the absolutely continuous and singular part of the measure $u_t$. Since $-u_t+ {H(x,\nabla u)}\le\alpha$ in the sense of measures, decomposing the two sides of the inequality into their absolutely continuous and singular part, we get
	$$
	-u_t^{ac}+ {H(x,\nabla u)}\le\alpha ,\qquad -u_t^s\le 0 .
	$$
	Let $\tilde\alpha:=(-u_t^{ac}+ {H(x,\nabla u)})\vee f(t,x,0)$. We have
	$$
	-u_t+ {H(x,\nabla u)}=-u_t^{ac}+ {H(x,\nabla u)}-u_t^s\le\tilde\alpha.
	$$
	Moreover, the inequality also holds in the sense of distributions as in \eqref{eq:distribution}. Hence $(u,\tilde\alpha)\in\mathscr K$. Further, because $\alpha\ge f(t,x,0)$, we also have $\tilde\alpha\le\alpha$, which implies $F^*(t,x,\tilde\alpha)\le F^*(t,x,\alpha)$ for a.e. $(t,x)\in Q_T$. This implies $\mathcal K(u,\tilde\alpha)\le \mathcal K(u,\tilde\alpha)$.
	
	Since $(u,\alpha)$ is optimal, we must have $F^*(t,x,\tilde\alpha)=F^*(t,x,\alpha)$. Recall that $F^*(t,x,s)$ is strictly increasing for $s\ge f(t,x,0)$, which implies $\tilde\alpha=\alpha$. Then
	$$
	\big(-u_t^{ac}+ {H(x,\nabla u)}\big)\vee f(t,x,0)=f(t,x,m).
	$$
	Thus, 
	$$
	 -u_t^{ac}+ {H(x,\nabla u)}=f(t,x,m)
	$$
	whenever $f(t,x,m)>f(t,x,0)$, i.e. when $m>0$. This concludes the proof.
\end{proof}

We have proved that the minima of the variational Problems \ref{var2} and \ref{var3} provide solutions for the Problem \ref{problem_mfg}, in the sense of Definition \ref{defmfg}. 
The subsequent proposition verifies the converse.
\begin{pro}\label{reverse}
	Consider the setting of Problems \ref{var2} and \ref{var3}. 
	Suppose Assumptions \ref{hp}, \ref{hp2}, and \ref{hp3} hold. Let $(u,m)\in BV(Q_T)\times L^p(Q_T)$ solve Problem \ref{problem_mfg}, in the sense of Definition \ref{defmfg}. Let $\alpha(t,x):=f(t,x,m(t,x))$ and $w(t,x):=-mH_p(t,x,\nabla u(t,x))$.  Then, $(u,\alpha)\in\mathscr K$ and $(m,w)\in\mathscr M$, and $(m,w)$ and $(u,\alpha)$ minimize, respectively, the functionals $\mathcal M$ and $\mathcal K$ in Problems \ref{var2} and \ref{var3}.
\end{pro}
\begin{proof}
	Let $(u,m)$ be a solution of Problem \ref{problem_mfg}. Let $\alpha=f(t,x,m)$ and $w=-m {H_p(x,\nabla u)}$.  We have $(u,\alpha)\in\mathscr K$. We must prove that $(u,\alpha)$ minimizes $\mathcal K$. Thanks to \eqref{FinalDuality}, it suffices to prove that $\mathcal I(v)\ge\mathcal K(u,\alpha)$ $\forall v\in\mathscr I$, where $\mathscr I$ is defined in \eqref{eq:Iset}. We have
	$$
	\mathcal I(v)=\intif F^*\big(t,x,-v_t+ {H(x,\nabla v)}\big) dxdt-\into v(0) dm_0(x)-\intb jv dxdt .
	$$
	The convexity of $F^*$ implies
	$$
	F^*(t,x, -v_t+ {H(x,\nabla v)})\ge F^*(t,x,\alpha)+m(-v_t+ {H(x,\nabla v)}-\alpha) ,
	$$
	where we also used $F^*_s(\alpha)=(F_m)^{-1}(\alpha)=m$. Hence,
	\begin{align*}
		\mathcal I(v)\ge\intif F^*(t,x,\alpha)-\into v(0) dm_0&-\intb jv dxdt\\
		&+\intif m\big(-v_t+ {H(x,\nabla v)}-\alpha\big)dxdt .
	\end{align*}
	Since $v\in\mathcal C^1(Q_T)$, we can use it as a test function for the Fokker-Planck equation. 
	Accordingly, from \eqref{generalformula}, we have
	\begin{align*}
		\into v(0) dm_0+\intb jv dxdt&+\intif mv_t dxdt\\&=\into m(T)\psi dx+\intif m {H_p(x,\nabla u)}\nabla v dxdt .
	\end{align*}
	Therefore,
	$$
	\mathcal I(v)\ge \intif F^*(t,x,\alpha) dxdt-\into m(T)\psi+\intif m\big( {H(x,\nabla v)}- {H_p(x,\nabla u)}\cdot\nabla v-\alpha\big)dxdt .
	$$
	From the definition of $L$, we have
	\begin{align*}
		\intif m\big( {H(x,\nabla v)}- {H_p(x,\nabla u)}\cdot\nabla v\big)dxdt=-&\intif m\left(-\frac wm\cdot\nabla v- {H(x,\nabla v)}\right)dxdt\\
		\ge-&\intif m {L\left(x,\frac wm\right)}dxdt .		
	\end{align*}
	Substituting into the previous inequality and using Item~\ref{solitem5} of Definition \ref{defmfg}, we get
	\begin{align*}
		\mathcal I(v)\ge&\intif F^*(t,x,\alpha) dxdt-\into m(T)\psi  dx-\intif m\alpha dxdt-\intif m {L\left(x,\frac wm\right)}dxdt\\
		=&\intif F^*(t,x,\alpha) dxdt-\into u(0)dm_0-\intb ju dxdt=\mathcal K(u,\alpha) ,
	\end{align*}
	which proves that $(u,\alpha)$ minimizes $\mathcal K$.\\
	
	Similar computations prove that $(m,w)$ is a minimum for $\mathcal M$. Given $(\mu,z)\in\mathscr M$, we can assume $\mathcal M(\mu,z)<+\infty$. Hence, $(\mu,z)\in L^s(Q_T;\R^{N+1})$ for some $s>1$, thanks to Proposition \ref{prop_regM}. We can use \eqref{eq:seconda} in Remark \ref{remark35} and write
	$$
	\mathcal M(\mu,z)=\intif \left[\mu   {L\left(x,\frac z\mu\right)}+F(t,x,\mu)\right]dxdt+\into\psi \mu(T,dx) .
	$$
	The convexity of $F$ implies
	$$
	F(t,x,\mu)\ge F(t,x,m)+f(t,x,m)(\mu-m)=F(t,x,m)+\alpha(\mu-m) .
	$$
	Substituting into $\mathcal M(\mu,z)$ and using \eqref{dual_leq} for the pairs $(u,\alpha)$, $(\mu,z)$, and \eqref{dual_eq} for the pairs $(u,\alpha)$, $(m,w)$, we get
	\begin{align*}
		\mathcal M(\mu,z)&=\intif\left[F(t,x,m)-\alpha m\right]dxdt+\intif \left[\alpha\mu+\mu  {L\left(x,\frac z\mu\right)}\right]dxdt+\into \psi \mu(T,dx)\\
		&\ge\intif\left[F(t,x,m)-\alpha m\right]dxdt+\into u(0)dm_0(x)+\intb ju dxdt\\
		&=\intif\left[m   {L\left(x,\frac wm\right)}+F(t,x,m)\right]dxdt+\into\psi m(T,dx)=\mathcal M(m,w) ,
	\end{align*}
	which proves that $(m,w)$ minimizes $\mathcal M$ and concludes the proof.
\end{proof}

%% file: uniqueness.tex
We conclude by establishing the uniqueness of solutions for Problem \ref{problem_mfg}, which follows directly from the uniqueness of Problems \ref{var1} and \ref{var2}.
Since the uniqueness for Problem \ref{var2} is proved in Theorem \ref{duality}, we now focus on proving the uniqueness of solutions for Problem \ref{var3}, from which the uniqueness of solutions for Problem \ref{problem_mfg} will follow.

We first generalize \eqref{dual_eq}.
\begin{lemma}
\label{lem81}
	Suppose Assumptions \ref{hp}, \ref{hp2}, \ref{hp3}, and \ref{hp4} hold. Let $(u,m)$ solve Problem \ref{problem_mfg}, in the sense of Definition \ref{defmfg}, and let $w=-m{H_p(x,\nabla u)}$, $\alpha=f(t,x,m)$. Then, for a.e. $t\in(0,T]$,
	\begin{equation}\label{gen_dual_eq}
		\into u(0) m_0 dx+\int_0^t\int_{\partial\Omega} ju dxds=\into u(t) m(t,dx)+\inti \left[m{L\left(x,\frac wm\right)}+ m\alpha\right]dxds .
	\end{equation}
\end{lemma}
As previously mentioned, equality \eqref{dual_eq} is formally obtained by using $u$ as a test function in the equation for $m$ and integrating over $Q_T$. Here,  we extend this result by proving the validity of identity on $Q_t=[0,t]\times\Omega$. This generalization plays a crucial role in establishing the uniqueness of solutions for Problem \ref{var3}, as shown in Theorem \ref{uniq_var3}.

\begin{proof}[Proof of Lemma 8.1]
	By Proposition \ref{prop_regM}, we have $(m,w)\in L^\theta(Q_T;\R^{N+1})$, for some $\theta>1$.
	We consider the approximation $(u^n,\alpha^n)\in\mathscr I\times\mathcal C^\infty$ provided by Corollary \ref{cor:strong_grad}.  By applying \eqref{generalformula} with $\xi=u^n$, we get for all $t\in[0,T]$,
	$$
	\into u^n(0) m_0 dx+\inti \big(mu^n_t+w\cdot\nabla u^n\big)dxds+\int_0^t\int_{\partial\Omega} ju^n dxds=\into u^n(t)  m(t,dx) .
	$$
	Since $u^n\to u$ in $L^1(W^{1,1})$, $u^n(0)\weak u(0)$ and $-u^n_t+{H(x,\nabla u^n)}\le\alpha^n$, we can argue as in Corollary \ref{cor:dual} to obtain
	$$
	\into u(0) dm_0+\int_0^t\int_{\partial\Omega} ju dxds\le\limsup_{n\to+\infty}\into u^n(t) m(t,dx)+\inti \left[m{L\left(x,\frac wm\right)}+ m\alpha\right]dxds .
	$$
	To handle the $\limsup$, we observe that $u^n$ is uniformly bounded in $L^1([0,T];W^{1,r}(\Omega))$ and $u^n_t$ is uniformly bounded in $L^1([0,T];L^1(\Omega))$. Thus, 
	by applying the compactness result in 
	 \cite[Corollary 4 and Theorem 5]{Simon}, we have that $u^n$ is relatively compact in 
	 $L^1([0,T];L^p(\Omega))$ for $p<\frac{rN}{N-r}$.
	 Therefore, up to a subsequence, $u^n\to u$ in $L^1([0,T];L^p(\Omega))$. Furthermore, $u^n(t)\to u(t)$ in $L^p(\Omega)$ for a.e. $t\in(0,T)$. From Assumption \ref{hp4}, we have
	$$
	r>\frac{Nq}{Nq+q-2};
$$
	that is, 
	$$
	\frac{rN}{N-r}>q'
	$$
	which implies
	$$
	v^n(t)\to v(t)\quad\text{in }L^{q'}(\Omega)\text{ for a.e. }t\in(0,T) .
	$$
	Because $m(t)\in L^q(\Omega)$ for a.e. $t\in(0,T)$, we have $\into u^n(t)m(t) dx\to\into u(t)m(t) dx$. Therefore,
	\begin{equation}\label{leq}
		\into u(0) dm_0+\int_0^t\int_{\partial\Omega} ju dxds\le\into u(t)m(t) dx+\inti \left[m{L\left(x,\frac wm\right)}+ m\alpha\right]dxds .
	\end{equation}
	Similarly, we get
	$$
	\into u(t)m(t) dx+\int_t^T\int_{\partial\Omega} ju dxds\le\into\psi  m(T,dx)+\intf \left[m{L\left(x,\frac wm\right)}+ m\alpha\right]dxds .
	$$
	Using the preceding inequality in \eqref{dual_eq}, we get
	$$
	\into u(0) dm_0+\int_0^t\int_{\partial\Omega} ju dxds\ge\into u(t)m(t) dx+\inti \left[m{L\left(x,\frac wm\right)}+ m\alpha\right]dxds .
	$$
	The preceding inequality, together with \eqref{leq}, implies \eqref{gen_dual_eq}.
\end{proof}

Now, we establish the uniqueness result for Problem \ref{var3}, adapting the ideas in \cite{Graber}.
\begin{teo}\label{uniq_var3}
	Consider the setting of Problem~\ref{var3}, and suppose Assumptions \ref{hp}, \ref{hp2}, \ref{hp3}, and \ref{hp4} hold. Then, if $(u,\alpha)\in\mathscr K$ and $(v,\beta)\in\mathscr K$ solve Problem \ref{var3}, we have $\alpha=\beta$ and $u=v$ a.e. in the set $\{\alpha>f(t,x,0)\}$.
\end{teo}
\begin{proof}
	Let $(u,\alpha)$, $(v,\beta)\in\mathscr K$ be two minimizers of $\mathcal K$. Consider the unique solution $(m,w)\in\mathscr M$ of Problem \ref{var2} (see Theorem \ref{duality}). By Theorem \ref{teo:ex_mfg}, we have $\alpha=\beta=f(t,x,m)$ in the set $\{m>0\}=\{\alpha>f(t,x,0)\}$, and $w=-m{H_p(x,\nabla u)}=-m{H_p(x,\nabla v)}$. We need to establish that $u=v$ in the set $\{m>0\}$.
	
	We first assume, without loss of generality, that $u \geq v$ almost everywhere; we will later remove this assumption. Under this assumption, since both $(u,\alpha)$ and $(v,\alpha)$ are minimizers with the same $\alpha$, we have
	\begin{equation}\label{equality}
		\into u(0) dm_0+\intb ju dxdt=\into v(0) dm_0+\intb jv dxdt .
	\end{equation}
	Consider the two solutions $(u,m)$ and $(v,m)$ of the MFG system. From \eqref{gen_dual_eq} in Lemma \ref{lem81}, we have, for a.e. $t\in(0,T]$,
	\begin{equation}\label{prima}
		\into u(0) m_0 dx+\int_0^t\int_{\partial\Omega} ju dxds=\into u(t) m(t,dx)+\inti \left[m{L\left(x,\frac wm\right)}+ m\alpha\right]dxds ,
	\end{equation}
	and
	\begin{equation}\label{seconda}
		\into v(0) m_0 dx+\int_0^t\int_{\partial\Omega} jv dxds=\into v(t) m(t,dx)+\inti \left[m{L\left(x,\frac wm\right)}+ m\alpha\right]dxds .
	\end{equation}
	Using \eqref{equality}, we observe that the left-hand sides of \eqref{prima} and \eqref{seconda} are equal. Thus, comparing the right-hand sides, we have
	$$
	\into u(t) m(t,dx)=\into v(t) m(t,dx)
	$$
	which implies
	$$
	\into \big[u(t)-v(t)\big] m(t,dx)=0 .
	$$
	Since the quantity in the integral is non-negative, this implies $u=v$ in $\{m>0\}$.
	
	To complete the proof, we need to remove the assumption that $u\geq v$. For arbitrary $u$ and $v$, we consider $z=u\vee v$, where $z(t,x)=\max\{u(t,x),v(t,x)\}$.. Then $z\ge u,v$. If we prove that $(z,\alpha)\in\mathscr K$ and minimizes the functional $\mathcal K$. The previous computations imply $z=u$ and $z=v$ in the set $\{m>0\}$, which means $u=v$ in $\{m>0\}$. Thus it remains to check that $(z,\alpha)\in\mathscr K$ and minimizes the functional $\mathcal K$.
	
	Since $z\ge u$ and $m_0,j\ge0$, we have $\mathcal K(z,\alpha)\le \mathcal K(u,\alpha)$. 
	Hence, we must prove $(z,\alpha)\in\mathscr K$.
	The regularity of $u$ and $v$  implies $(z,\alpha,\nabla z)\in BV(Q_T)\times L^{q'}(Q_T)\times L^r(Q_T)$. We want to prove that \eqref{eq:distribution} holds for the couple $(z,\alpha)$, i.e. for all $\phi\in\mathcal C_c^\infty((0,T]\times\overline\Omega)$, $\phi\ge0$ we have
	$$
	\intif \Big(z\phi_t+{H(x,\nabla z)}\phi\Big) dxdt\le\intif\alpha\phi dxdt+\into \psi(x)\phi(T,x) dx .
	$$
	We consider the approximating functions $(u^\eps,v^\eps,\alpha^\eps)$ given by Lemma \ref{lem:regularization}, and we define $z^\eps:=u^\eps\vee v^\eps$. Then, we have $z^\eps\to z$ in $L^1(W^{1,1})$, $z^\eps(T)=\psi$ and $-z^\eps_t+H(\nabla z^\eps)\le \alpha^\eps$, where the inequality holds a.e.. This implies
	\begin{align*}
		\intif \Big(z\phi_t+{H(x,\nabla z)}\phi\Big) dxdt=\lim_{\eps\to0}&\intif \Big(z^\eps\phi_t+{H(x,\nabla z^\eps)}\phi\Big) dxdt\\
		=\lim_{\eps\to0}&\intif \phi\Big(\!\!-z^\eps_t+{H(x,\nabla z^\eps)}\Big) dxdt+\into\psi(x)\phi(T,x) dx\\
		\le\lim_{\eps\to0}&\intif \alpha^\eps\phi dxdt+\into\psi(x)\phi(T,x) dx\\
		=&\intif \alpha\phi dxdt+\into\psi(x)\phi(T,x) dx ,
	\end{align*}
	which proves $(z,\alpha)\in\mathscr K$ and concludes the Theorem.
\end{proof}
The uniqueness of the Mean Field Games system solutions can be derived from the preceding Theorem and Theorem \ref{duality}. We report it here to complete the proof of the main theorem.

\begin{proof}[Proof of Theorem \ref{thm:main}]
	The existence part is proved in Theorem \ref{teo:ex_mfg}. 
	Regarding uniqueness, suppose that $(u,m)$ and $(v,\mu)$ are two solutions to Problem \ref{problem_mfg}. By Proposition \ref{reverse}, the pairs $(m,-m H_p(x,\nabla u))$ and $(\mu,-\mu H_p(x,\nabla v))$ solve Problem \ref{var2}. Since, by Theorem \ref{duality}, the solution to Problem \ref{var2} is unique, it follows that $m=\mu$ almost everywhere. Furthermore, $(u,f(t,x,m))$ and $(v,f(t,x,m))$ are minimizers of Problem \ref{var3}. Applying Theorem \ref{uniq_var3}, we conclude that $u=v$ almost everywhere in the set $\{m>0\}$. This completes the proof of the theorem.
\end{proof}